\documentclass[review]{elsarticle}
 \biboptions{sort&compress}
\usepackage[utf8]{inputenc}
  
\usepackage{hyperref}
\usepackage{xcolor}
\usepackage{mathtools,amsfonts,amsthm,amssymb}
\usepackage{subcaption,float}

\newtheorem{thm}{Theorem}[section]
\newtheorem{theorem}[thm]{Theorem}

\newtheorem{lem}[thm]{Lemma}
\newtheorem{lemma}[thm]{Lemma}
\newtheorem{cor}[thm]{Corollary}

\allowdisplaybreaks

\newtheoremstyle{hdef}{1em}{0em}{}{}{\bfseries}{.}{.5em}{\thmname{#1}\thmnumber{ #2}\thmnote{ (\hspace{-.01pt}{#3})}}
\theoremstyle{hdef}

\newtheorem{dfn}[thm]{Definition}

\newtheorem{rem}[thm]{Remark}

\makeatletter
\newtheoremstyle{premark}{1em}{0em}{
\addtolength{\@totalleftmargin}{1.5em}
\addtolength{\linewidth}{-1.5em}
\parshape 1 1.5em \linewidth}{}{\scshape}{.}{.5em}{}
\makeatother

\theoremstyle{premark}

\journal{Journal of \LaTeX\ Templates}

\begin{document}

\begin{frontmatter}

\title{Stieltjes Bochner spaces and applications to the study of parabolic equations\tnoteref{mytitlenote}}
\tnotetext[mytitlenote]{The authors were partially supported by Xunta de Galicia, project ED431C 2019/02, and by projects MTM2015-65570-P and MTM2016-75140-P of MINECO/FEDER (Spain).}

\author{
 Francisco J. Fernández\\
 \normalsize\emph{ e-mail:} fjavier.fernandez@usc.es \\
F. Adri\'an F. Tojo\\
\normalsize\emph{ e-mail:} fernandoadrian.fernandez@usc.es\\
\normalsize \textit{Instituto de Ma\-te\-m\'a\-ti\-cas, Facultade de Matem\'aticas,} \\ \normalsize\textit{Universidade de Santiago de Com\-pos\-te\-la, Spain.}
}





\begin{abstract}
This work is devoted to the mathematical analysis of 
Stieltjes Bochner spaces and their applications to the resolution of a 
parabolic equation with Stieltjes time derivative. This novel formulation 
allows us to study parabolic equations that present impulses at certain 
times or lapses where the system does not evolve at all and presents an 
elliptic behavior. We prove several theoretical results related to existence 
of solution, and propose a full algorithm for its computation, illustrated with 
some realistic numerical examples related to population dynamics.
\end{abstract}


\end{frontmatter}


\section{Introduction}

The main goal of this work is to analyze the existence of solution of the partial differential equation
\begin{equation} \label{eq:estado1}
\begin{dcases}
u_g' -\nabla \cdot (k_1 \nabla u) + k_2 u = f, & \text{in}\; 
[(0,T)\setminus C_g] \times \Omega, \vspace{0.1cm} \\
u = 0,& \text{on}\; (0,T) \times \partial \Omega, \vspace{0.1cm} \\
u(0,x) = u_0(x), & \text{in}\; \Omega,
\end{dcases}
\end{equation}
where $\Omega \subset \mathbb{R}^3$ is a domain with a smooth enough 
boundary $\partial \Omega$ and $u_g'$ is the Stieltjes derivative in some Banach space $V$ with respect to a left-continuous nondecreasing function
$g:\mathbb{R} \rightarrow \mathbb{R}$. This is, given a function
$u:[0,T] \rightarrow V$, we define for each $t \in [0,T] \setminus C_g$, 
$u_g'(t)$ as the following limit in $V$
in the case it exists:
\begin{equation} \label{eq:gcontinuidad}
u_g'(t):=\begin{dcases}
 \lim_{s\to t} \frac{u(s)-u(t)}{g(s)-g(t)}, & \text{if}\; t \notin D_g, 
\\ 
 \frac{u(t^+)-u(t)}{g(t^+)-g(t)}, & \text{if}\; t \in D_g,
\end{dcases}
\end{equation}
where 
\begin{equation}
D_g=\{s \in \mathbb{R}:\; g(s^+)-g(s)>0\}
\end{equation} 
and 
\begin{equation}
C_g=\{s \in \mathbb{R}:\; 
g \text{ is constant on } (s-\varepsilon,s+\varepsilon) \text{ for some } \varepsilon\in{\mathbb R}^+\}.
\end{equation}

The study of this type of derivatives and its application to the field of ODEs 
appears in \cite{POUSO2015,POUSO2017,POUSO2018,FrigonTojo}. We use the 
notation established in previous works. We further assume that $k_1>0$ is a 
positive constant, $k_2 \geq 0$, $u_0 \in L^2(\Omega)$ and 
$f \in L^2_g([0,T],L^2(\Omega))$, with $([0,T],\mathcal{M}_g,\mu_g)$ a suitable 
measure space associated to $g$ \cite{POUSO2015}. It is important to mention 
 that if $u:A \subset \mathbb{R} 
\rightarrow H$ is $g$-continuous for every $t_0 \in A$ in the sense of:
\begin{equation}
\forall \varepsilon >0\, \exists \delta >0\; :\; \left[t \in A,\; |g(t)-g(t_0)|<\delta \right] 
\Rightarrow \|u(t)-u(t_0)\|_H<\varepsilon,
\end{equation}
then $f$ is constant in the same intervals as $g$ \cite[Proposition 3.2]{POUSO2017}. Moreover, continuity in the previous sense does not imply continuity in the classical sense, but if $g$ is continuous at 
$t_0 \in [0,T]$, then so is $f$ \cite{POUSO2015}. Taking into account that $g$ es left-continuous, we observe that the spaces of bounded $g$-continuous functions ${\mathcal B}{\mathcal C}_g([0,T],L^2(\Omega))$ and ${\mathcal B}{\mathcal C}_g([0,T),L^2(\Omega))$ are basically the same since any function in ${\mathcal B}{\mathcal C}_g([0,T],L^2(\Omega))$ must be continuous at $T$.

 Observe that the 
Stieltjes derivative is not defined at the points of $C_g$. The connected 
components of $C_g$ correspond 
to lapses when our system does not evolve at all and presents an 
elliptic behavior. The set $D_g$ of discontinuities of $g$ correspond 
with times when sudden changes occur and which are usually 
introduced in the form of impulses. Finally, in the remaining set of times 
$[0,T]\setminus (C_g \cup D_g)$ the system presents a parabolic 
behavior and the different slopes of the derivator $g$ (see \cite{POUSO2018}) 
correspond to different influences of the corresponding times, namely, 
the bigger the slope of $g$ the more important the corresponding 
times are for the process. In a certain sense, system~\eqref{eq:estado1} 
can be considered as a degenerate parabolic system.

The main difficulty in the mathematical 
analysis of system~\eqref{eq:estado1} lies in the fact that we cannot 
consider the distributional derivative in time for defining the concept of 
solution. Thence, we will define the solution in terms of its integral 
representation and prove new Lebesgue-type differentiation results 
in order to recover the Stieltjes derivative $g$-almost everywhere in $[0,T]$. 
Results proven in the appendix of \cite{BREZIS1973} suggest  that 
it might be possible to define the concept of $g$-distributional 
derivative, thus proving the relationship between the $g$-absolute continuous 
functions and the $W^{1,1}_g$-type spaces. It is important to mention that 
in the case where $g(t)=t$, we recover the standard derivative, so 
all of the results that we will prove  extend the classical theory.

In this work we will establish the basis 
of the mathematical analysis for system~\eqref{eq:estado1} as well as a 
first numerical approximation of its solution. In order to organize the contents 
of the paper, we will divide the work in the following sections: In Section 2 we 
will introduce the Stieltjes-Bochner spaces in which we will define the concept 
of solution. We will also prove new Lebesgue-differentiation--type results for the 
Stieltjes derivatives and some continuous injections. In Section~3 we will 
define the concept of solution of problem~\eqref{eq:estado1}. In Section 4 
we will prove an existence result for system~\eqref{eq:estado1} that generalizes 
some aspects of the classical theory of parabolic partial differential equations. 
Finally, in Section 5, we will present a realistic example and we will propose a 
numerical scheme. In this example we will have a parabolic-elliptical behavior, 
showing the advantage of considering derivatives of the Stieltjes type.

\section{Stieltjes Bochner spaces}

We start by defining the spaces in which to look for the solution of the problem 
and its fundamental properties. In order to achieve this, and for convenience of 
the reader, we start by reviewing 
some concepts related to Bochner spaces 
\cite{YOSIDA1995,SHOWALTER1997,BREZIS1973,leoni2009}. Let us consider 
the measure space $(\mathbb{R},\mathcal{M}_g,\mu_g)$ induced by 
$g$ \cite{POUSO2015} and $V$ a real Banach space.

\begin{dfn}[$\mathbf g$-measurable functions] Given $f:\mathbb{R} \rightarrow V$ we say:
\begin{itemize}
\item $f$ is a \emph{simple $g$-measurable function} if there exits a finite 
set $\{x_k\}_{k=1}^n \subset V$ such that $A_k=f^{-1}(\{x_k\})\in \mathcal{M}_g$, with 
$\mu_g(A_k)<\infty$ and $f=\sum_{k=1}^n x_k \chi_{A_k}$. In this case, we its 
\emph{integral} as
 \begin{equation}
 \int f(s) \, \operatorname{d} \mu_g(s)=\sum_{k=1}^n x_k \mu_g(A_k) \in V.
 \end{equation}
\item $f$ is a \emph{strongly $g$-measurable function} 
 (or simply $g$-measurable function) if there exists a sequence $\{f_n\}_{n \in \mathbb{N}}$ of simple $g$-measurable functions such that
 $f_n(s)\rightarrow f(s)$ in $V$ for $g$-a.e. $s \in \mathbb{R}$.
 \item $f$ is a \emph{weakly $ g$-measurable function} if $s \in \mathbb{R} 
 \rightarrow v(f(s)) \in \mathbb{R}$ is $g$-measurable for every $v \in V'$.
 \end{itemize}
\end{dfn}

 Pettis' Theorem (cf. \cite[\S V.4, Theorem 1]{YOSIDA1995}) establishes that a function 
$f:\mathbb{R} \rightarrow V$ is strongly $g$-measurable if and only if it is weakly $ g$-measurable and 
$g$-almost separably-valued. Therefore if we consider a separable Banach space $V$ both concepts 
are equivalent.

Now we define the concept of a $g$-integrable $V$-valued function.

\begin{dfn}[$\mathbf g$-integrable $\mathbf V$-valued function] A $g$-measurable function
 $f:\mathbb{R} \rightarrow V$ is said to be a \emph{$g$-integrable $V$-valued function} if there exists a sequence of simple $g$-measurable functions $\{\varphi_n\}_{n \in \mathbb{N}}$ such that $\varphi_n(t)\to f$ in $V$ for $g$-a.e. $t\in\mathbb{R}$ and
 \begin{equation}
 \lim_{n \to \infty} \int \|\varphi_n(s)-f(s)\|_{V} \, \operatorname{d} \mu_g(s)=0.
 \end{equation}
 The integral of $f$ in $B \in \mathcal{M}_g$ is defined as
 \begin{equation}
 \int_B f(s) \, \operatorname{d} \mu_g(s) = \lim_{n \to \infty} \int \chi_B(s) \varphi_n(s) \, 
 \operatorname{d} \mu_g(s) \in V.
 \end{equation}
\end{dfn}

Bochner's Theorem (cf. \cite[\S V.5, Theorem 1]{YOSIDA1995}) allows us to 
characterize the $g$-integrable $V$-valued functions in terms of the $g$-integrability of its norm, 
that is a $g$-measurable function $f:\mathbb{R} \rightarrow V$ is $g$-integrable if and only if 
 \begin{equation}
 \int \|f(s)\|_V \, \operatorname{d}\mu_g(s) < \infty,
 \end{equation}
 and, in such a case,
 \begin{equation}
 \left\|\int_B f(s) \, \operatorname{d} \mu_g(s) \right\|_V \leq \int_B \|f(s)\|_V \, \operatorname{d} \mu_g(s),
 \end{equation}
 for every $B \in \mathcal{M}_g$.

Furthermore, we have the following lemma that we will allow us to establish the concept 
of solution for our problem. From now on, given $v\in V$ and $w\in V'$ we will write $\left<w,v\right>:=w(v)$.

\begin{lemma}[{\cite[\S V.5, Corollary 2]{YOSIDA1995}}] Let $W$ be a Banach space. $T:V\to W$ a bounded linear operator. Then, if $f:\mathbb{R} 
 \rightarrow V$ is $g$-integrable, we have that $T\circ f:\mathbb{R} \rightarrow 
 W$ is $g$-integrable and
 \begin{equation}
 \int_B (T\circ f)(s)\, \operatorname{d} \mu_g(s)=T\left( \int_B f(s) \, \operatorname{d} \mu_g(s)\right) ,\; \forall B \in \mathcal{M}_g.
 \end{equation}
 In particular, for $f:\mathbb{R} 
 \rightarrow V'$ and $v\in V$, 
 \begin{equation}
 \int_B \left<f(s),v\right> \, \operatorname{d} \mu_g(s)=\left<\int_B f(s) \, \operatorname{d} \mu_g(s),v\right>.
 \end{equation}
\end{lemma}

\begin{dfn}[$\mathbf{L^p_g}$ spaces] With the usual equivalence relation functions which are equal $g$-a.e., we define, for $1\leq p < \infty$, the space $L^p_g([0,T],V)$ as the set of $g$-measurable functions $f:[0,T]\to V$ such that
 \begin{equation}
 \int_{[0,T)} \|f(s)\|_V^p \, \operatorname{d} \mu_g(s) < \infty. 
 \end{equation}
 Analogously, we define the space $L_g^{\infty}([0,T],V)$ of those functions which are essentially bounded. 
\end{dfn}

\begin{rem} We have that the set $L^p_g([0,T],V)$ with $1\leq p \leq \infty$ is a
 Banach space with the norm
 \begin{equation}
 \|f\|_{L^p_g([0,T],V)}=\left\{\begin{array}{ll}
 \left[\int_{[0,T)} \|f(s)\|_{V}^p\, \operatorname{d} \mu_g(s)\right]^{\frac{1}{p}}, & 1\leq p <\infty, 
 \vspace{0.1cm} \\ 
 \operatorname*{sup~ess}_{t \in [0,T]} \|f(t)\|_{V}, & p=\infty,
 \end{array}\right.
 \end{equation}
\end{rem}
--see \cite[Theorem 8.15]{leoni2009}.

From now on, let $V$ be a real reflexive separable Banach space and let $H$ be a Hilbert space such that $V$ continuously and densely embedded in $H$. Identifying $H$ with its dual $H'$ we have that $V\subset H \equiv H' \subset V'$. 

Now we will adapt \cite[Theorem 2, p. 134]{YOSIDA1995} to our setting (see Theorem~\ref{td}) to guarantee that an indefinite Bochner $g$-integral is
$g$-differentiable. This result will be fundamental in order to recover the 
existence of $g$-derivative $g$-almost everywhere for the solutions of 
problem~\eqref{eq:estado1}. In order to check this we present some previous definitions and results.

\begin{thm}[{\cite[Theorem 2.4]{POUSO2015}}]\label{PoRo} Assume that $f:[0,T]\to\overline {\mathbb R}$ is integrable on $[0,T]$ with respect to $\mu_g$ and consider its indefinite Lebesgue-Stieltjes integral
 \begin{equation}F(t)=\int_{[0, t)} f(s)\, \operatorname{d} \mu_{g}(s) \quad \text { for all } t \in[0,T].\end{equation}
 Then there is a $g$-measurable set $N\subset[0,T]$ such that $\mu_g(N)=0$ and
 \begin{equation}F_{g}^{\prime}(t)=f(t) \quad \text { for all } t \in[0,T] \setminus N.\end{equation}
\end{thm}

\begin{dfn}Let $X$, $Y$ be vector spaces. An operator $L:X\to Y$ is said to be of \emph{finite rank} if $L(X)$ is contained in a finite dimensional vector subspace of $Y$.
\end{dfn}
Observe that any simple $g$-measurable function is of finite rank. The extension of the 
previous theorem to finite rank functions is straightforward, so we have the following theorem.

\begin{thm}[Generalized Lebesgue's differentiation Theorem for finite rank functions] 
 \label{ftc2}
 Let $f:[0,T]\to  V$ is a Bochner $g$-integrable finite rank function 
 and consider its indefinite Lebesgue-Stieltjes integral
 \begin{equation}F(t)=\int_{[0, t)} f(s) \, \operatorname{d} \mu_{g}(s) \in V \quad \text { for all } t \in[0,T].\end{equation}
 Then there is a $g$-measurable set $N\subset[0,T]$ such that $\mu_g(N)=0$ and
 \begin{equation}F_{g}^{\prime}(t)=f(t) \in V\quad \text { for all } t \in[0,T] \setminus N.\end{equation}
\end{thm}

\begin{theorem}[Generalized Lebesgue's differentiation Theorem] \label{td} Let $T\in\mathbb R^+$ and
 $f:[0,T]\rightarrow V$ be a Bochner $g$-integrable function 
 and consider its indefinite Lebesgue-Stieltjes integral
 \begin{equation}
 F:t \in [0,T] \rightarrow \int_{[0,t)} f(s) \, \operatorname{d} \mu_g(s) \in V.
 \end{equation}
 Then there exists a $g$-measurable set $N\subset [0,T]$ such that
 $\mu_g(N)=0$ and
 \begin{equation}
 F_g'(t)=f(t),\; \forall t \in [0,T] \setminus N.
 \end{equation} 
\end{theorem}

\begin{proof} Let us consider the sequence of simple $g$-measurable functions 
 $\{f_n\}_{n \in \mathbb{N}}$ such that
 \begin{itemize}
 \item $ \|f_n(s)\|_{V} \leq \|f(s)\|_V \left(1+\frac{1}{n} \right)$.
 \item $ \lim\limits_{n \to \infty} f_n(s) =f(s)$, $g$-a.e. $s \in [0,T]$.
 \end{itemize}
 Let $t\in [0,T]\setminus(C_g\cup D_g)$. For every $s\in [0,T]$, $s\ne t$, we have that $g(s)\ne g(t)$ and we can consider, assuming that $s>t$,
 \begin{align*}
 & \frac{F(s)-F(t)}{g(s)-g(t)}-f(t)= \frac{1}{g(s)-g(t)}\int_{[t,s)}[f(\xi)-f(t)]\operatorname{d}\mu_g(\xi)\\= & \frac{1}{g(s)-g(t)}\int_{[t,s)}[f(\xi)-f_n(\xi)+f_n(\xi)-f_n(t)]\operatorname{d}\mu_g(\xi)+f_n(t)-f(t).
 \end{align*}
 Thus,
 \begin{align*}
& \left\|
  \frac{F(s)-F(t)}{g(s)-g(t)}-f(t) \right\|_{V}  \le  \frac{1}{g(s)-g(t)} \int_{[s,t)} \|f(\xi)-f_n(\xi)\|_{V} \, \operatorname{d}\mu_g(\xi) \\
 & + \left\|\frac{1}{g(s)-g(t)}\int_{[t,s)}[f_n(\xi)-f_n(t)]\operatorname{d}\mu_g(\xi)\right\|_V +\left\|f_n(t)-f(t)\right\|_V.
 \end{align*}
 Let us define
 \begin{equation}
 \widetilde{u}_n: t \in [0,T] \rightarrow \widetilde{u}_n(t)=\int_{[0,t)} \|f(\xi)-f_n(\xi)\|_{V} \, \operatorname{d} \mu_g(\xi).
 \end{equation}
 It is clear that $\widetilde{u}_n \in L^1_g([0,T])$. Hence, we can use Theorem~\ref{PoRo} to conclude that
 \begin{equation}
 \lim_{s \to t^+} \frac{\widetilde{u}_n(s)-\widetilde{u}_n(t)}{g(s)-g(t)} = \|f(t)-f_n(t)\|_{V}.
 \end{equation}
 Since $f_n$ is finite rank we can use Theorem~\ref{ftc2} so, in the topology of $V$,
 \begin{equation}\lim_{s\to t^+}\frac{1}{g(s)-g(t)}\int_{[t,s)}[f_n(\xi)-f_n(t)]\operatorname{d}\mu_g(\xi)=0.\end{equation}
 Finally,
 \begin{equation}
  \lim_{s \to t^+ } 
 \left\|
 \frac{F(s)-F(t)}{g(s)-g(t)}-f(t) \right\|_{V}\leq 2\left\|f(t)-f_n(t)\right\|_V.
  \end{equation}
 Hence, taking $n \to \infty$, we obtain the desired result. The case $s<t$ is analogous.
\end{proof}

We denote by $\mathcal{C}_g([0,T],V)$ the set of $g$-continuous 
 functions on interval $[0,T]$ in the sense of~\eqref{eq:gcontinuidad}, and by 
 $\mathcal{BC}_g([0,T],V)$ the subset of bounded $g$-continuous functions on $[0,T]$. We have that 
 the space $\mathcal{BC}_g([0,T],V)$ equipped with the supremum norm
 \begin{equation}
 \|h\|_{0}=\sup_{t \in [0,T]} \|h(t)\|_{V}, \; \forall h \in \mathcal{C}_g([0,T],V),
 \end{equation}
 is a Banach space. The proof is analogous to one given in \cite[Theorem 3.4]{POUSO2017}. 

Given $1\leq p,q \leq \infty$, we define 
\begin{equation}
\begin{array}{c}
\displaystyle
\widetilde{W}_g^{1,p,q}([0,T],V,V'):=\Big\{
u \in L^p_g([0,T],V)\ :\ \exists\, \widetilde{u} \in L^q_g([0,T],V'),
\vspace{0.1cm}\\ \displaystyle
 u(t)=u(0)+\int_{[0,t)} 
\widetilde{u}(s)\, 
\operatorname{d} \mu_g(s)\in V',\, t\in[0,T]\Big\}.
\end{array}
\end{equation}

\begin{rem} Observe that given $u \in \widetilde{W}_g^{1,p,q}([0,T],V,V')$, 
 $\widetilde{u} \in L^q_g([0,T],V')$ is unique up to a set of $g$-measure zero. To 
 see this assume there are two, $\widetilde u_1$ and $\widetilde u_2$, such functions. Then, 
 for every $v\in V'$,
 \begin{equation}\int_{[0,t)} [v(\widetilde u_2(s))-v(\widetilde u_1(s))] \, \operatorname{d} \mu_g(s)=0.
 \end{equation}
 Now, using Theorem~\ref{td} and differentiating on both sides, $v(\widetilde u_2(s))-v(\widetilde u_1(s))=0$ for $g$-a.e. $t\in [0,T]$ and every $v\in V'$, so $\widetilde u_2(t)=\widetilde u_2(t)$ $g$-a.e. Furthermore, by Theorem~\ref{td}, $\widetilde u=u'_g$ $g$-a.e.
\end{rem}

If we endow the space $\widetilde{W}_g^{1,p,q}([0,T],V,V')$ with the norm
\begin{equation}
\|u\|_{\widetilde{W}^{1,p,q}_g([0,T],V,V')}=\|u\|_{L^p_g([0,T],V)}+\|u'_g\|_{L^q_g([0,T],V')},
\end{equation}
it is clear that $\left( \widetilde{W}^{1,p,q}_g([0,T],V,V'),\|\cdot\|_{\widetilde{W}^{1,p,q}_g([0,T],V,V')}\right) $ is a normed vector space.

\begin{lemma}\label{lempq} Given $1\leq p,q \leq \infty$ we get the following continuous inclusion
 \begin{equation}
 \left( \widetilde{W}^{1,p,q}_g([0,T],V,V'),\|\cdot\|_{\widetilde{W}^{1,p,q}_g([0,T],V,V')} \right) 
 \hookrightarrow \left( \mathcal{BC}_g([0,T],V'),\|\cdot\|_0\right) .
 \end{equation}
\end{lemma}

\begin{proof} Let $u \in \widetilde{W}^{1,p,q}_g([0,T],V,V')$ and  define 
 \begin{equation}v: t \in [0,T] \rightarrow v(t)=\int_{[0,t)} \widetilde{u}(s)\, \operatorname{d} \mu_g(s).\end{equation}
  We have that 
 $v \in \mathcal{AC}_g([0,T],V')$. Moreover,
 \begin{equation}
 \begin{aligned}
 \|v(t)\|_{V'} \leq & \int_{[0,t)} \|\widetilde{u}(s)\|_{V'}\, \operatorname{d} \mu_g(s) \\
 = &
 \|\widetilde{u}\|_{L^1_g([0,T],V')} \leq \mu_g([0,T])^{1-\frac{1}{q}} \|\widetilde{u}\|_{L^q_g([0,T],V')},
 \end{aligned}
 \end{equation}

 where $\mu_g([0,T])^{1-\frac{1}{q}}$ is the embedding constant of $L^q_g([0,T],V')$ into\\$L^1_g([0,T],V')$ --cf. \cite[Theorem 13.17]{hewitt1965}. 
 Thus,
 \begin{equation}\label{redb}
 \begin{aligned}
 & \|v\|_{L^p_g([0,T],V')}=  \left( \int_{[0,T)} \|\widetilde{u}(s)\|_{V'}^p\, \operatorname{d} \mu_g(s)\right) ^{\frac{1}{p}}\\ & \le \left( \int_{[0,T)} (\mu_g([0,T))^{1-\frac{1}{q}}\|\widetilde{u}\|_{L^q_g([0,T],V')})^p\, \operatorname{d} \mu_g(s)\right) ^{\frac{1}{p}} \\ & =  \mu([0,T))^{1+\frac{1}{p}-\frac{1}{q}}\|\widetilde{u}\|_{L^q_g([0,T],V')},
 \end{aligned}
 \end{equation}

 Now, $v=u+c$ with $c=u(0) \in V'$ and, if $\widetilde N$ is the embedding constant of $V$ in $V'$, using~\eqref{redb},
 \begin{equation}
 \begin{aligned}
 \|c\|_{V'} = & \mu_g([0,T))^{-\frac{1}{p}} \left[ \int_{[0,T)} \|c\|_{V'}^p \, \operatorname{d} \mu_g(s) \right]^{\frac{1}{p}}  \\ 
 \leq & \mu_g([0,T))^{-\frac{1}{p}} \|v-u\|_{L^p_g([0,T],V')} \\
 \leq & \mu_g([0,T))^{-\frac{1}{p}} \|u\|_{L^p_g([0,T],V')} + \mu_g([0,T))^{-\frac{1}{p}} 
 \|v\|_{L^p_g([0,T],V')} \\ \leq & \widetilde N \mu_g([0,T))^{-\frac{1}{p}} \|u\|_{L^p_g([0,T],V)}
 +\mu_g([0,T))^{1-\frac{1}{q}}\|\widetilde{u}\|_{L^q_g([0,T],V')}.
 \end{aligned}
 \end{equation}
 Finally,
 \begin{equation}
 \begin{aligned}
 \|u\|_{0} = & \sup_{t \in [0,T]} \left\| \int_{[0,t)} \widetilde{u}(s) \, 
 \operatorname{d} \mu_g(s)-c \right\|_{V'} \\
 \leq & \int_{[0,T)} \|\widetilde{u}(s)\|_{V'} \, \operatorname{d} \mu_g(s)+\|c\|_{V'} \\
 \leq & 2\mu_g([0,T))^{1-\frac{1}{q}} \|\widetilde{u}\|_{L^q_g([0,T],V')}+\widetilde N \mu_g([0,T))^{-\frac{1}{p}} \|u\|_{L^p_g([0,T],V)}.
 \end{aligned}
 \end{equation}
 Thus,
 \begin{equation}
 \|u\|_{0} \leq \max\left\{2\mu_g([0,T))^{1-\frac{1}{q}},\widetilde N \mu_g([0,T))^{-\frac{1}{p}}\right\} \|u\|_{\widetilde{W}^{1,p,q}_g([0,T],V,V')}.
 \end{equation}

\end{proof}
\begin{cor} The space $(\widetilde{W}^{1,p,q}_g([0,T],V,V'),\|\cdot\|_{\widetilde{W}^{1,p,q}_g([0,T],V,V')} )$ 
 is a Banach space.
\end{cor}

\begin{proof} We first prove that $\widetilde{W}^{1,p,q}_g([0,T],V,V')$ is a Banach space. Consider a Cauchy sequence
 $\{u_n\}_{n \in \mathbb{N}}$ in $\widetilde{W}^{1,p,q}_g([0,T],V,V')$. In particular, 
 the sequences $\{u_n\}_{n \in \mathbb{N}}$ and
 $\{\widetilde{u}_n\}_{n \in \mathbb{N}}$ are Cauchy sequences in $L^p_g([0,T],V)$ and $L^q_g([0,T],V')$. Furthermore, thanks to Lemma~\ref{lempq}, they will also be so in
 $\mathcal{BC}_g([0,T],V')$. Since the previous spaces are complete, there will exist $u \in L^p_g([0,T],V)$ and $\widetilde{u}\in L^q_g([0,T],V')$ such that
 $u_n \rightarrow u$ in $L^p_g([0,T],V)$, $\widetilde{u}_n \rightarrow \widetilde{u}$ strongly in $L^q_g([0,T],V')$, $u_n(t) \rightarrow u(t)$ strongly in $V'$ for every $t \in [0,T]$. Thus, 
 we can take the following expression to the limit for every $ v \in V'$ and every $t \in [0,T]$,
 \begin{equation}
 \begin{aligned}
 \left<u_n(t),v\right>_{V',V} = & \left<u_n(0),v\right>_{V',V}
 + \int_{[0,t)} \left<\widetilde{u}_n(s),v\right>_{V',V} \, \operatorname{d} \mu_g(s),
 \end{aligned}
 \end{equation}
 and we get, for every $ v \in V'$ and every $t \in [0,T]$,
 \begin{equation}
 \left<u(t),v\right>_{V',V} = \left<u(0),v\right>_{V',V} \\
 \displaystyle
 + \int_{[0,t)} \left<\widetilde{u}(s),v\right>_{V',V} \, \operatorname{d} \mu_g(s).
 \end{equation}
 Since
 \begin{equation}
 \left<u_n(t),v\right>_{V',V}-\left<u(t),v\right>_{V',V} \leq 
 \|u_n(t)-t(t)\|_{V'}\|v\|_V, 
 \vspace{0.2cm} 
 \end{equation}
 we have that,for every $ v \in V$ and every $ t \in [0,T]$,
 \begin{equation}
 \begin{aligned}
& \int_{[0,t)} 
  \left(\left<\widetilde{u}_n(s),v\right>_{V',V}\, -
 \left<\widetilde{u}(s),v\right>_{V',V}\right) \operatorname{d} \mu_g(s) \\
 \leq  &
 \int_{[0,t)} \|\widetilde{u}_n(s)-\widetilde{u}(s)\|_{V'}\|v\|\, \operatorname{d} \mu_g(s) 
 \\ 
 \leq  &
 \|v\|\, \int_{[0,T)} \|\widetilde{u}_n(s)-\widetilde{u}(s)\|_{V'} \operatorname{d} \mu_g(s).
 \end{aligned}
 \end{equation}
 Therefore $\{u_n\}_{n \in \mathbb{N}}\to u$ in  $\widetilde{W}^{1,p,q}_g([0,T],V,V')$.
\end{proof}

Now we are going to prove that the space $\widetilde{W}^{1,p,q}_g([0,T],V,V')$ 
is also reflexive. In order to achieve that, we need some results that we are going 
to review for the convenience of the reader.

\begin{dfn}[{\cite{Bennett}}]
 A pair $(X, Y)$ of Banach spaces $X$ and $Y$ is called a 
 \emph{compatible couple} if there is some Hausdorff topological vector space in which each of $X$ and $Y$ is continuously embedded. Let $((X,\|\cdot\|_X),(Y,\|\cdot\|_Y))$ be a compatible couple, Then $X\cap Y$ with the norm $\|x\|=\max\{\|x\|_X,\|x\|_Y\}$ and $X+Y:=\{x+y\ :\ x\in X,\ y\in Y\}$ with the norm \begin{equation}\|z\|_{X+Y}=\inf_{\substack{x\in X\\ y\in Y\\x+y=z}}\left( \|x\|_X+\|y\|_Y\right) \end{equation} are Banach spaces. The cartesian product $X\times Y$ with the norm $\|(x,y)\|_{X\times Y}=\|x\|_X+\|y\|_Y$ is such that
 $X+Y\simeq (X\times Y)/L$ where $L:=\{(z,-z)\in X\cap Y\}$. A compatible couple $(X, Y)$ with the property that $X\cap Y$ is dense in $X$ and in $Y$ is called a \emph{conjugate couple}.
\end{dfn}

\begin{lem}[{\cite[Theorem 3.1, p. 15]{Krein}}]\label{lemcc} If $(X, Y)$ is a conjugate couple, then $(X\cap Y)'$ is isometric to $X'+ Y'$ and $(X+ Y)'$ is isometric to $X'\cap Y'$.
\end{lem}

Let us define the  space
\begin{equation}
\begin{aligned}
\widetilde L^q_g([0,T],V'):= & \Big\{ u:[0,T]\to V'\ :\ \exists\, \widetilde{u} \in L^q_g([0,T],V'),\\
\vspace{0.1cm} &
 u(t)=u(0)+\int_{[0,t)}\widetilde{u}(s)\, 
\operatorname{d} \mu_g(s) \in V',\, \forall t\in[0,T]\Big\}.
\end{aligned}
\end{equation}
It is clear that $\widetilde L^q_g([0,T],V')\subset \mathcal{BC}_g([0,T],V')$ 
so, $\forall u \in \widetilde L^q_g([0,T],V')$, $u(0)$ has sense in $V'$. 
We also have $u'_g(t)=\widetilde{u}(t) \in V'$, $g$-a.e. $t \in [0,T]$.

\begin{lem}$\widetilde L^q_g([0,T],V')$ is a Banach space with the norm
\begin{equation}\|u\|_{\widetilde L^q_g([0,T],V')}:=\|u(0)\|_{V'}+\|\widetilde u\|_{L^q_g([0,T],V')}.\end{equation}
\end{lem}

\begin{proof}
First, $\|\cdot\|_{\widetilde L^q_g([0,T],V')}$ is a norm. It is clearly subadditive and absolutely homogeneous. It is left to check that it is positive definite. If $\|u\|_{\widetilde L^q_g([0,T],V')}=0$ then $\|u(0)\|_{V'}=0$ and $\|\widetilde u\|_{L^q_g([0,T],V')}=0$. Since they both are norms, $u(0)=0$ and $\widetilde u=0$. By definition of $u$,
\begin{equation}
u(t)=u(0)+\int_{[0,t)} 
\widetilde{u}(s)\, 
\operatorname{d} \mu_g(s)=0.
\end{equation}

Now, take a Cauchy sequence $\{u_n\}_{n\in\mathbb N}$ in $\widetilde L^q_g([0,T],V')$. Then
\begin{equation}
u_n(t)=u_n(0)+\int_{[0,t)} 
\widetilde{u}_n(s)\, 
\operatorname{d} \mu_g(s).
\end{equation}
Thus, $\{u_n(0)\}_{n\in\mathbb N}$ and $\{\widetilde{u}_n\}_{n\in\mathbb N}$ are Cauchy sequences and, since both $V'$ and $L^q_g([0,T],V')$ are Banach spaces, they converge to $x$ and $v$ respectively. Now, define
\begin{equation}
u(t)=x+\int_{[0,t)} 
v(s)\, 
\operatorname{d} \mu_g(s).
\end{equation}
Clearly $u\in\widetilde L^q_g([0,T],V')$ and $\{u_n\}_{n\in\mathbb N}\to u$ in $\widetilde L^q_g([0,T],V')$. Hence, we have that $\widetilde L^q_g([0,T],V')$ is a Banach space.
\end{proof}
\begin{lem} $\widetilde{W}^{1,p,q}_g([0,T],V,V')$ and $ L^p_g([0,T],V)\cap\widetilde L^q_g([0,T],V')$ are isomorphic.
\end{lem}
\begin{proof} To see this remember that
\begin{equation}\|u\|_{\widetilde{W}^{1,p,q}_g([0,T],V,V')}=\|u\|_{L^p_g([0,T],V)}+\|u'_g\|_{L^q_g([0,T],V')},\end{equation}
and
\begin{equation}
\|u(0)\|_{V'}\le\|u\|_{0} \leq C \|u\|_{\widetilde{W}^{1,p,q}_g([0,T],V,V')},\end{equation}
where $C:=\max\left\{2\mu_g([0,T))^{1-\frac{1}{q}},\widetilde N \mu_g([0,T))^{-\frac{1}{p}}\right\}$. Hence, 
\begin{equation}
\begin{aligned}
& \|u\|_{L^p_g([0,T],V)\cap\widetilde L^q_g([0,T],V')} \\= & 
\max\{\|u\|_{L^p_g([0,T],V)},\|u\|_{\widetilde L^q_g([0,T],V')}\}\\ = & \max\{\|u\|_{L^p_g([0,T],V)},\|u(0)\|_{V'}+\|u'_g\|_{L^q_g([0,T],V')}\}\\ \le & \max\{\|u\|_{\widetilde{W}^{1,p,q}_g([0,T],V,V')},C \|u\|_{\widetilde{W}^{1,p,q}_g([0,T],V,V')}+\|u\|_{\widetilde{W}^{1,p,q}_g([0,T],V,V')}\}\\  =& (C+1)\|u\|_{\widetilde{W}^{1,p,q}_g([0,T],V,V')}. 
\end{aligned}
\end{equation}

On the other hand,
\begin{equation}
\begin{aligned}
& \|u\|_{\widetilde{W}^{1,p,q}_g([0,T],V,V')}= 
\|u\|_{L^p_g([0,T],V)}+\|u'_g\|_{L^q_g([0,T],V')}\\ \le & 2\max\{\|u\|_{L^p_g([0,T],V)},\|u'_g\|_{L^q_g([0,T],V')}\}\\ \le & 2\max\{\|u\|_{L^p_g([0,T],V)},\|u(0)\|_{V'}+\|u'_g\|_{L^q_g([0,T],V')}\}\\ = &2\|u\|_{L^p_g([0,T],V)\cap\widetilde L^q_g([0,T],V')}. 
\end{aligned}
\end{equation}
\end{proof}

\begin{lem}
$\widetilde L^q_g([0,T],V')$ is reflexive.
\end{lem}
\begin{proof} Take the map
\begin{equation}
\begin{array}{rcl}
\varphi: V' \times L_g^q([0,T],V') & \rightarrow & \widetilde L^q_g([0,T],V') \\
\vspace{0.1cm}
(x,v) & \rightarrow & \varphi(x,v)=u,
\end{array}
\end{equation}
where, $\forall t \in [0,T]$, 
\begin{equation}
u(t)=x+\int_{[0,t)} v(s) \operatorname{d} \mu_g(s)\in V.
\end{equation}
Clearly, $\varphi$ is an isometric isomorphism with inverse
\begin{equation}
\begin{array}{rcl}
\varphi^{-1}:  \widetilde L^q_g([0,T],V') & \rightarrow & V' 
\times L^q_g({[}0,T{]},V') \\ \vspace{0.1cm}  
u & \rightarrow & \varphi^{-1}(u)=(u(0),u'_g).
\end{array}
\end{equation}
 $\varphi$ induces the isomorphism $\varphi^*:\widetilde L^q_g([0,T],V')' \to (V'\times L^q_g({[}0,T{]},V'))'$.
 We know that $(X\times Y)'=X'\times Y'$ with the norm $\|(f,g)\|_{X'\times Y'}=\max\{\|f\|_{X'},\|g\|_{Y'}\}$ --and vice-versa, see \cite[p. 14]{Krein}. Hence,  thanks to Riesz representation theorem (\cite[Theorem 8.17]{leoni2009}), we have that $\widetilde L^q_g([0,T],V')'$ is isomorphic to 
 $V \times L^{q^*}_g([0,T],V)$, where $p^*\in[1,\infty]$ such that $p+p^*=pp^*$, with the norm 
 \begin{equation}
 \|(f,g)\|_{V'\times L^{q^*}_g([0,T],V')}=\max\{\|f\|_{V'},\|g\|_{L^{q^*}_g([0,T],V')}\}.
 \end{equation}
 Taking the dual again, we obtain that $\widetilde L^q_g([0,T],V')$ is reflexive.
\end{proof}
\begin{lem}$(L^p_g([0,T],V),\widetilde L^q_g([0,T],V'))$ is a conjugate couple.
\end{lem} 
\begin{proof}
 Observe that we have the continuous inclusion
 \begin{equation}\label{i1}
 L^p_g([0,T],V)\cap\widetilde L^q_g([0,T],V')\simeq\widetilde{W}^{1,p,q}_g([0,T],V,V')\hookrightarrow\mathcal{BC}_g([0,T],V').
 \end{equation}
 Therefore, $(L^p_g([0,T],V),\widetilde L^q_g([0,T],V'))$ is a compatible couple when embedded in $L_g^p([0,T],V')$. Since we have the dense embeddings
 \begin{equation}
\widetilde{W}^{1,p,q}_g([0,T],V,V') \hookrightarrow L^p_g([0,T],V)
 \end{equation}
and 
 \begin{equation} 
\widetilde{W}^{1,p,q}_g([0,T],V,V') \hookrightarrow \widetilde L^q_g([0,T],V'),
 \end{equation}
  $(L^p_g([0,T],V),\widetilde L^q_g([0,T],V'))$ is a conjugate couple. 
\end{proof}

\begin{cor} $\widetilde{W}^{1,p,q}_g([0,T],V,V')$ is reflexive.
\end{cor}
\begin{proof}
 Using Lemma~\ref{lemcc} and the fact that $(L^p_g([0,T],V),\widetilde L^q_g([0,T],V'))$ is a conjugate couple,
 \begin{equation}\widetilde{W}^{1,p,q}_g([0,T],V,V')'= L^p_g([0,T],V)'+\widetilde L^q_g([0,T],V')'.\end{equation}
 Thus,
 \begin{equation}
 \begin{array}{rcl}
 \widetilde{W}^{1,p,q}_g([0,T],V,V')''&=&
 L^p_g([0,T],V)''\cap\widetilde L^q_g([0,T],V')'' \\ \vspace{0.1cm}
&=&  L^p_g([0,T],V)\cap\widetilde L^q_g([0,T],V'),
\end{array}
\end{equation}
 and so, $\widetilde{W}^{1,p,q}_g([0,T],V,V')$ is reflexive.
\end{proof}

\section{The concept of solution}

In this section we will establish the concept of solution of 
system~\eqref{eq:estado1}. In order to properly motivate this 
concept, we will proceed by analogy with the classic case. So, 
we consider the following system:
\begin{equation} \label{eq:clasical}
\begin{dcases}
 \frac{\partial u}{\partial t} -\nabla \cdot (k_1 \nabla u) + k_2 u = f, & \text{in}\; 
(0,T) \times \Omega, \vspace{0.1cm} \\
u = 0,& \text{on}\; (0,T) \times \partial \Omega, \vspace{0.1cm} \\
u(0,x) = u_0(x), & \text{in}\; \Omega,
\end{dcases}
\end{equation}
with $f \in L^2([0,T],L^2(\Omega))$, $u_0 \in L^2(\Omega)$. If we denote by 
\begin{equation}
W^{1,p,q}([0,T],V,V')=\left\{ u \in L^p([0,T],V)\ :\ 
\frac{\operatorname{d} u}{\operatorname{d} t} \in L^q([0,T],V')\right\},
\end{equation}
where $\frac{\operatorname{d} u}{\operatorname{d} t}$ is distributional derivative of $u$. We have 
that there exists an unique element $u \in W^{1,2,2}([0,T],H_0^1(\Omega),
H^{-1}(\Omega))\cap \mathcal{C}([0,T],L^2(\Omega))$ (where $H_0^1(\Omega) = 
\{u \in W^{1,2}(\Omega):\; u|{_{\partial \Omega}}=0\}$ and 
$H^{-1}(\Omega)=H_0^1(\Omega)'$) such that 
$u(0)=u_0$ and, for every $ v \in H_0^1(\Omega)$, $u$ satisfies the following 
variational formulation in $\mathcal{D}'(0,T)$:
\begin{equation}
\frac{\operatorname{d}}{\operatorname{d} t} \left(u(t),v\right)_{L^2(\Omega)} +
k_1 \int_{\Omega} \nabla u(t) \cdot \nabla v \, \operatorname{d} x
+
k_2 \int_{\Omega} u(t) v\, \operatorname{d} x = 
\int_{\Omega} f(t) v\, \operatorname{d} x,
\end{equation} 
Thus, the distributional derivative is such that
\begin{equation}
\frac{\operatorname{d} u}{\operatorname{d} t}=f+\nabla \cdot (k_1 \nabla u)-k_2 u \in L^2([0,T],H^{-1}(\Omega)),
\end{equation}
and then, by  \cite[Proposition A.6]{BREZIS1973}, we can identify $u$ with 
an element of the space $\widetilde{W}^{1,2,2}([0,T],H_0^1(\Omega),H^{-1}(\Omega)) =$
\begin{equation}
\begin{array}{c}
\displaystyle
 
\Big\{u \in L^2([0,T],H_0^1(\Omega))\ :\ 
\exists \widetilde{u} \in L^2([0,T],H^{-1}(\Omega)),\\ 
\vspace{0.1cm} 
\displaystyle
u(t)=u(0)+\int_{[0,t]} \widetilde{u}(s) \operatorname{d} s \in H^{-1}(\Omega),
\, \forall t \in [0,T] \Big\}
\end{array}
\end{equation}
with $\widetilde{u}=\frac{\operatorname{d} u}{\operatorname{d} t}$ almost everywhere in $[0,T]$. So, we have that the 
spaces \begin{equation}\widetilde{W}^{1,2,2}([0,T],H_0^1(\Omega),H^{-1}(\Omega))\quad \text{and} \quad W^{1,2,2}([0,T],H_0^1(\Omega),H^{-1}(\Omega))\end{equation} are essentially the same and, for every 
$ t \in [0,T]$, and $v \in H_0^1(\Omega)$ we have
\begin{equation} \label{eq:solinicial}
\left<u(t),v\right>=
\left<u_0,v\right>+\int_{[0,t]} \int_{\Omega} f(s)\,v- k_1\nabla u(s) \cdot \nabla v-
k_2 u(s)\, v \, \operatorname{d} x \, \operatorname{d} s,
\end{equation}
Moreover, thanks to the Lebesgue Differentiation Theorem \cite[Theorem 1.6]{SHOWALTER1997}, there exists
\begin{equation}
\lim_{h \rightarrow 0} \frac{u(t+h)-u(t)}{h}=f(t)+\nabla \cdot (k_1 \nabla u(t))-k_2 u(t) \in 
H^{-1}(\Omega),
\end{equation}
for a.e. $t\in [0,T]$. That is, there exists the classical derivative 
in time, $u'(t)\in H^{-1}(\Omega)$, almost everywhere in $[0,T]$ and it 
satisfies 
\begin{equation}
u'(t)=f(t)+\nabla \cdot (k_1 \nabla u(t))-k_2 u(t) \in 
H^{-1}(\Omega),\; \text{for a.e.}\; t \in [0,T].
\end{equation}

In our case, we cannot define the space 
$W^{1,2,2}_g([0,T],H_0^1(\Omega),H^{-1}(\Omega))$ because we 
don't have a $g$-distributional derivative. Still, we can use the generalized 
Lebesgue’s Differentiation Theorem (Theorem~\ref{td}) and define the solution of 
system~\eqref{eq:estado1} in the following way.

\begin{dfn}[Solutions of system~\eqref{eq:estado1}] \label{definition1} Given 
 $u_0\in L^2(\Omega)$ and $f \in L^2_g([0,T],L^2(\Omega))$, we say that
 \begin{equation}
 u \in \widetilde{W}_g^{1,2,2}([0,T],H^1_0(\Omega),H^{-1}(\Omega)) \cap 
 \mathcal{BC}_g([0,T],L^2(\Omega))
 \end{equation}
 is a solution of equation~\eqref{eq:estado1} if, for every $ v \in V$ and every 
 $ t \in [0,T]$,
 \begin{equation} \label{eq:estado2}
 \begin{aligned}
 (u(t),v) = & 
 (u_0,v)_H +\int_{[0,t)} (f(s),v) \, \operatorname{d}\mu_g(s) \vspace{0.2cm} \\ & 
 -k_1 \int_{[0,t)}(\nabla u(s),\nabla v) \, \operatorname{d} \mu_g(s)- 
 k_2 \int_{[0,t)} (u(s),v) \, \operatorname{d}\mu_g(s).
 \end{aligned}
 \end{equation}
\end{dfn}

In the following corollary, a direct consequence of Theorem~\ref{td}, we will see that we can 
recover the $g$ derivative of 
the solution $g$-almost everywhere in time.

\begin{cor} If $u$ is a solution of equation~\eqref{eq:estado1} then there exists a $g$-measurable set, $N\subset [0,T]$, with 
 $\mu_g(N)=0$ such that 
 \begin{equation}
 u_g'(t)=f(t) + \nabla \cdot (k_1 \nabla u)(t) - k_2 u(t) \in 
 H^{-1}(\Omega),\; \forall t \in [0,T]\setminus N.
 \end{equation}
 \end{cor}
\section{An existence result}

In this section we will study the existence and uniqueness of solution of the equation~\eqref{eq:estado1} where $\Omega \subset \mathbb{R}^3$ is a domain with a sufficiently regular boundary
$\partial \Omega$. We take
$V=H_0^1(\Omega)$ and $H=L^2(\Omega)$ in the functional framework of the previous section and we use the classical diagonalization method --see \cite{DAUTRAY1992}-- in order to prove existence of solution. The fundamental goal is to recover those results known for the case $D_g=\varnothing$.

Let us establish some notation. Let $\{w_k\}_{k \in \mathbb{N}}$ be an eigenvector basis of $H_0^1(\Omega)$, orthonormal with respect to $L^2(\Omega)$, related to the following spectral problem
\begin{equation}
\begin{dcases}
w_n \in H_0^1(\Omega), \; \forall n \in \mathbb{N}, \\
k_1 \int_{\Omega} \nabla w_n \cdot \nabla v \, \operatorname{d} x + 
k_2 \int_{\Omega} w_n v \, \operatorname{d} x= \lambda_n 
\int_{\Omega} w_n v \, \operatorname{d} x, \; \forall v \in H_0^1(\Omega),\; 
\forall n \in \mathbb{N},
\end{dcases}
\end{equation}
where $0<\lambda_1<\lambda_2 <\cdots\to\infty$ and such that $\left\{ {w_n}/{\sqrt{\lambda_n}} \right\}_{n\in
 \mathbb{N}}$ is a basis $H_0^1(\Omega)$ for the scalar product of $H_0^1(\Omega)$, 
\begin{equation}\label{eq:escalar}
(u,v)= \int_{\Omega} (k_1 \nabla u \cdot \nabla v + k_2 u v ) \operatorname{d} x.
\end{equation}
Let $\Delta g(t)=g(t^+)-g(t)$ for any function $g$. For our next result we recall the following.
\begin{thm}[{\cite[Lemma 6.4]{POUSO2017}}]\label{teosol}
 Let $x_0\in{\mathbb R}$, $h,d\in L^1_g([a,b))$, with $d(t)\Delta g(t)\ne 1$ for every $t\in[a,b)\cap D_g$ and $\sum_{t\in[a,b)\cap D_g}|\ln|1-d(t)\Delta g(t)||<\infty$. Then
\begin{equation}
\begin{dcases}
x'_g(t)+d(t)x(t)= h(t)& \text{for }g\text{-a.e. in } [a,b),\\
x(a) = x_0.& 
\end{dcases}
\end{equation} 
 has a unique solution $x\in{\mathcal A}{\mathcal C}_g([a,b])$.
 \end{thm}
\begin{theorem} [Existence of solution of the system~\eqref{eq:estado1}] 
 Let $u_0\in L^2(\Omega)$, $f \in L^2_g([0,T],L^2(\Omega))$ and
 $g$ a nondecreasing function, continuous in a neighborhood of $t=0$ and left-continuous in $(0,T]$.
 Assume that, for every $n\in{\mathbb N}$,
\begin{itemize}
 \item[\rm{(H1)}]
 $\displaystyle\lambda_n \Delta g(t) \neq 1 ,\; \forall t \in [0,T] \cap D_g$, $\displaystyle\sum_{t\in[a,b)\cap D_g}|\ln|1-\lambda_n\Delta g(t)||<\infty$,
 \item[\rm{(H2)}]
 $\displaystyle e^{-2 \lambda_n \,\mu_g([0,t)\setminus D_g)} 
 \prod_{u \in [0,t) \cap D_g} |1-\lambda_n \Delta g(u)|^2 \leq C_1, \;\forall t \in [0,T]
 $,
\item[\rm{(H3)}] $\displaystyle \int_{[0,T)} \lambda_n 
 e^{-2 \lambda_n \,\mu_g([0,t)\setminus D_g)} 
 \prod_{u \in [0,t) \cap D_g} |1-\lambda_n \Delta g(u)|^2\, \operatorname{d}\mu_g(t) 
 \leq C_1$,
 \item[\rm{(H4)}]
 $\displaystyle \int_{[0,t)} e^{-2 \lambda_n \,\mu_g([s,t)\setminus D_g)} 
 \frac{\prod_{u \in [s,t) \cap D_g} |1-\lambda_n \Delta g(u)|^2}{
 |1-\lambda_n \Delta g(s)|^2} \,\operatorname{d}\mu_g(s) \leq C_2,
 \: \forall t \in [0,T]$,
 \item[\rm{(H5)}]
 $\displaystyle \int_{[0,T)} \int_{[0,t)} \lambda_n e^{-2 \lambda_n \,\mu_g([s,t)\setminus D_g)} 
 \frac{\prod_{u \in [s,t) \cap D_g} |1-\lambda_n \Delta g(u)|^2}{
 |1-\lambda_n \Delta g(s)|^2} \,\operatorname{d}\mu_g(s)\, \operatorname{d} \mu_g(t)$\vspace*{.7em}\\$\leq C_2$, 
\end{itemize}
 where $C_1,C_2\in{\mathbb R}^+$ are constants. Then there exists
 \begin{equation}
 u \in \widetilde{W}_g^{1,2,2}([0,T],H^1_0(\Omega),H^{-1}(\Omega)) \cap 
 \mathcal{BC}_g([0,T],L^2(\Omega)),
 \end{equation}
 unique solution of the equation~\eqref{eq:estado1} in the sense of the formulation
~\eqref{eq:estado2}, such that satisfies the following bounds with respect to the data
 \begin{equation}\label{eq:acotaec}
 \begin{aligned}
 & \|u\|_{L^{\infty}_g([0,T],L^2(\Omega))}+
 \|u\|_{L^2_g([0,T],H^1_0(\Omega))}
 +\|u_g'\|_{L_g^{2}([0,T],H^{-1}(\Omega))}
 \\ \leq & \widehat{C}_1 \|u_0\|_{L^2(\Omega)}+
 \widehat{C}_2 \|f\|_{L_g^{2}([0,T],L^2(\Omega))},
 \end{aligned}
 \end{equation}
 where $\widehat{C}_1,\widehat{C}_2\in{\mathbb R}^+$ are constants.
\end{theorem}

\begin{proof} In order to make a clearer proof, we will divide it into five parts. In the first part we will approximate problem~\eqref{eq:estado2} using the functions of the spectral basis. Then, in Part 2, we will obtain bounds for the solutions associated to the discrete problem. In the third part, we will analyze how to take the limit and recover a solution of the continuous problem. Later, in Part 4, we will analyze the continuity with respect to the data. In the last part we will prove the uniqueness of solution.

 \textbf{\textbullet\ Part 1, spectral basis approximation.}
 
 Given $k \in \mathbb{N}$, 
 let us write
 \begin{equation}
 u_0^k=(u_0,w_k), \quad
 f^k(t)=(f(t),w_k),
 \end{equation}
 for every $k=1,\ldots,n$. We have that
 \begin{equation}
 u_0=
 \sum_{k=1}^{\infty} u_0^k w_k, \quad
 f(t)=
 \sum_{k=1}^{\infty} f^k(t) w_k.
 \end{equation}
 We look for a solution of the form
 \begin{equation}
 u(t)=\sum_{k=1}^{\infty} \xi^k(t) w_k,
 \end{equation}
 which, substituting formally in~\eqref{eq:estado2},
 \begin{equation} \label{eq:galerkin1}
 \begin{aligned}
 (u(t),w_k)=&(u_{0},w_k)+\int_{[0,t)} (f(s),w_k)\,\operatorname{d}\mu_g(s)
 \\&-k_1 \int_{[0,t)} \int_{\Omega} \nabla u(s) \cdot \nabla w_k \, \operatorname{d} x \, \operatorname{d}\mu_g(s)
 \\
 & -
 k_2
 \int_{[0,t)} \int_{\Omega} (u(s), w_k) \, \operatorname{d} x \, \operatorname{d}\mu_g(s),\; \forall t \in [0,T], \; 
 \; k \in \mathbb{N},
 \end{aligned}
 \end{equation}
 it will satisfy, for every $k\in \mathbb{N}$, the approximated problem
 \begin{equation} \label{eq:galerkin2}
 \xi^k(t)=u_{0}^k+\int_{[0,t)}[ f^k(s)
 - \lambda_k \, \xi^k(s)
 ]\operatorname{d}\mu_g(s) ,\; \forall t \in [0,T].
 \end{equation}
 Thanks to the Fundamental Theorem of Calculus for the Lebesgue-Stieltjes integral (see \cite[Theorem 5.1]{POUSO2017}) that the previous problem is equivalent to
 \begin{equation} \label{eq:galerkin3}
 \begin{dcases}
 (\xi^k)_g'(t) + \lambda_k \,\xi^k(t)=f^k(t),\; 
 \text{for $g$-almost all $t \in [0,T]$}, \\
 \xi^k(0)=u_{0}^k, \; k\in \mathbb{N}.
 \end{dcases}
 \end{equation}
 Now, by Theorem~\ref{teosol}, if the compatibility conditions (H1) are satisfied, there exists a unique solution 
 $\xi^k \in \mathcal{AC}_g([0,T])\cap \mathcal{BC}_g([0,T])$, $k=1,\ldots,n$, which, furthermore, we can compute explicitly taking into account the exponential function in \cite[Lemma 6.4]{POUSO2017}, this is,
 \begin{equation} \label{eq:galerkin4}
 \xi^k(t)=\widehat{e}^{k}(t)^{-1} u_{0}^k +
 \widehat{e}^{k}(t)^{-1} \int_{[0,t)} \widehat{e}^{k}(s) \widetilde{f}^k(s) \,\operatorname{d}\mu_g(s),
 \end{equation}
 with, for every $k=1,\ldots,n$, 
 \begin{equation}
 \widehat{e}^k(t)=\begin{dcases}
 e^{\int_{[0,t)} \widehat{d}^k(s) \,\operatorname{d}\mu_g(s)}, & 0\leq t \leq t_1, \\
 (-1)^j e^{\int_{[0,t)} \widehat{d}^k(s) \,\operatorname{d}\mu_g(s)}, & t_j < t \leq t_{j+1}, \; 
 k=1,\ldots,N_k, 
 \end{dcases}
 \end{equation}
 with $t_{N_k+1}=T$ and $\widehat{d}^k$ given by
 \begin{equation}
 \widehat{d}^k(t)=\begin{dcases}
 \widetilde{d}^k(t), & \; t \in [0,T] \setminus D_g, \\
 \frac{\ln|1+\widetilde{d}^k(t)\Delta g(t)|}{\Delta g(t)}, &
 t \in [0,T] \cap D_g,
 \end{dcases}
 \end{equation}
 where 
 \begin{equation}\label{dft}
 \widetilde{d}^k(t)=
 \frac{\lambda_k}{1-\lambda_k\Delta g(t)}, \quad 
 \widetilde{f}^k(t)=
 \frac{f^k(t)}{1-\lambda_k\Delta g(t)},
 \end{equation}
 are functions in $L^1_g([0,T])$, as it was pointed out in the proof of \cite[Proposition~6.8]{POUSO2017}. Finally, 
 the points $\{t_1,\ldots,t_{N_k}\}$, with $k=1,\ldots,n$, are those in
 \begin{equation}T_{\widetilde{d}^k}^-=\{ t \in [0,T]\cap D_g: \; 1+ \widetilde{d}^k(t)\Delta g(t)<0\}.\end{equation}
 This is a finite set because
 \begin{equation}
 \sum_{t \in T_{\widetilde{d}^k}^-} 1 < \sum_{t \in T_{\widetilde{d}^k}^-} 
 |\widetilde{d}^k(t)\Delta g(t)| \leq \|\widetilde{d}^k\|_{L^1_g(0,T)} < \infty.
 \end{equation}
 Observe that, given $t \in [0,T]\cap D_g$, we have that
 \begin{equation}
 \begin{aligned}
 \widehat{d}^k(t)=&
 \frac{1}{\Delta g(t)} \ln\left|1+\frac{\lambda_k \Delta g(t) }{1-\lambda_k\Delta g(t)} \right|
 =
 \frac{1}{\Delta g(t)} \ln \left| \frac{1}{1-\lambda_k\Delta g(t)} \right|\\
 =& 
 - \frac{1}{\Delta g(t)} \ln\left| 1-\lambda_k\Delta g(t) \right|. 
 \end{aligned}
 \end{equation}
 Thence,
 \begin{equation}\label{dsim}
 \widehat{d}^k(t)=\begin{dcases}
 \lambda_k, & t \in [0,T] \setminus D_g, \\
 - \frac{\ln\left| 1-\lambda_k\Delta g(t) \right|}{\Delta g(t)} , &
 t \in [0,T] \cap D_g.
 \end{dcases}
 \end{equation}
 Observe that, for a given $t \in [0,T] \cap D_g$, there exists an index $k$ 
 from which $\ln\left| 1-\lambda_k\Delta g(t) \right|$ is a strictly positive number. 

 \textbf{\textbullet\ Part 2: obtaining of bounds related to the solution of the discrete problem.}

 In what follows we will obtain a series of bounds of the solution of
~\eqref{eq:galerkin4} of the approximated problem~\eqref{eq:galerkin3}. Taking the absolute value on~\eqref{eq:galerkin4} we have that
 \begin{equation}
 |\xi^k (t)| \leq \frac{|u_0^k|}{|\widehat{e}^{k}(t)|} + 
 \int_{[0,t)} \frac{|\widehat{e}^{k}(s)|}{|\widehat{e}^{k}(t)|} |\tilde{f}^k(s)|\,\operatorname{d}\mu_g(s),
 \end{equation}
 Thence, taking into account Hölder's inequality and the parallelogram law,
 \begin{equation}\label{paralaw}
 |\xi^k (t)|^2 \leq 2 \frac{|u_0^k|^2}{|\widehat{e}^{k}(t)|^2} + 
 2 \int_{[0,t)} \frac{|\widehat{e}^{k}(s)|^2}{|\widehat{e}^{k}(t)|^2 |1-\lambda_k \Delta g(s)|^2} \,\operatorname{d}\mu_g(s)
 \int_{[0,t)} |f^k(s)|^2\,\operatorname{d}\mu_g(s).
 \end{equation}
 Now,
 \begin{equation}\label{cotae}
 \begin{aligned}
 \frac{1}{|\widehat{e}^{k}(t)|^2} = &
 \exp\left( {-2 \int_{[0,t)\setminus D_g} \widehat{d}^k(u) \,\operatorname{d}\mu_g(u)}\right) 
 \exp\left( {-2 \int_{[0,t) \cap D_g} \widehat{d}^k(u) \,\operatorname{d}\mu_g(u)}\right) 
 \\ 
 = &e^{-2 \lambda_k \,\mu_g([0,t)\setminus D_g)}
 \exp\left( {2 \sum_{u \in [0,t) \cap D_g}\ln|1-\lambda_k \Delta g(u)|}\right) 
 \\ 
 = & 
 e^{-2 \lambda_k \,\mu_g([0,t)\setminus D_g)} 
 \prod_{u \in [0,t) \cap D_g} |1-\lambda_k \Delta g(u)|^2.
 \end{aligned}
 \end{equation}

 Analogously, using~\eqref{cotae},
 \begin{equation}
 \begin{aligned}
 &\int_{[0,t)}
 \frac{|\widehat{e}^{k}(s)|^2}{|\widehat{e}^{k}(t)|^2|1-\lambda_k \Delta g(s)|^2} \,\operatorname{d}\mu_g(s) \\ = & 
 \int_{[0,t)} \frac{e^{-2 \lambda_k \,\mu_g([0,t)\setminus D_g)} 
 \prod_{u \in [0,t) \cap D_g} |1-\lambda_k \Delta g(u)|^2}{
 e^{-2 \lambda_k \,\mu_g([0,s)\setminus D_g)} 
 \prod_{u \in [0,s) \cap D_g} |1-\lambda_k \Delta g(u)|^2|1-\lambda_k \Delta g(s)|^2}\operatorname{d}\mu_g(s) \\
 = &
 \int_{[0,t)} e^{-2 \lambda_k \,\mu_g([s,t)\setminus D_g)} 
 \frac{\prod_{u \in [s,t) \cap D_g} |1-\lambda_k \Delta g(u)|^2}{
 |1-\lambda_k \Delta g(s)|^2} \,\operatorname{d}\mu_g(s).
 \end{aligned}
 \end{equation}
 Thus, from~\eqref{paralaw} and using (H2) and (H4) we have that, for every $k \in \mathbb{N}$ and every $t \in [0,T]$,
 \begin{equation} \label{eq:galerkin6}
 \begin{aligned}
 |\xi^k (t)|^2  \leq 2 \left[e^{-2 \lambda_k \,\mu_g([0,t)\setminus D_g)} 
 \prod_{u \in [0,t) \cap D_g} |1-\lambda_k \Delta g(u)|^2\right] \,|u_0^k|^2 \\ 
 + 2 \left[\int_{[0,t)} e^{-2 \lambda_k \,\mu_g([s,t)\setminus D_g)} 
 \frac{\prod_{u \in [s,t) \cap D_g} |1-\lambda_k \Delta g(u)|^2}{
  |1-\lambda_k \Delta g(s)|^2} \,\operatorname{d}\mu_g(s) \right]\\  \cdot \int_{[0,t)}|f^k(s)|^2 \,\operatorname{d}\mu_g(s) 
 \leq 2 C_1 \, |u_0^k|^2 + 2 C_2 \, \|f^k\|_{L_g^{2}([0,T],{\mathbb R})}^2.
 \end{aligned}
 \end{equation}
 On the other hand, from~\eqref{paralaw} and using (H3) and (H5) we have that
 \begin{equation} \label{eq:galerkin7}
 \begin{aligned}
 & \int_{[0,T)} \lambda_k |\xi^k(t)|^2\, \operatorname{d} \mu_g(t) \\ & \leq 2\left[
 \int_{[0,T)} \lambda_k
 e^{-2 \lambda_k \,\mu_g([0,t)\setminus D_g)} 
 \prod_{u \in [0,t) \cap D_g} |1-\lambda_k \Delta g(u)|^2\, \operatorname{d} \mu_g(t)\right]\,
 |u_0^k|^2 \\ &
 + 2 \int_{[0,T)} \int_{[0,t)} \lambda_k e^{-2 \lambda_k \,\mu_g([s,t)\setminus D_g)} 
 \frac{\prod_{u \in [s,t) \cap D_g} |1-\lambda_k \Delta g(u)|^2}{
 |1-\lambda_k \Delta g(s)|^2} \,\operatorname{d}\mu_g(s)\, \operatorname{d} \mu_g(t) \\ & \cdot
 \|f^k\|_{L_g^{2}([0,T],{\mathbb R})}^2 \leq 2 C_1 \, |u_0^k|^2 + 2 C_2 \, \|f^k\|_{L_g^{2}([0,T],{\mathbb R})}^2.
 \end{aligned}
 \end{equation}
 From~\eqref{eq:galerkin6} we deduce that, given
 $n,p\in \mathbb{N}$,
 \begin{equation} \label{eq:galerkin8}
 \sum_{k=n}^{n+p} |\xi^k(t)|^2 \leq 
 2C_1 \sum_{k=n}^{n+p} |u_0^k|^2 +
 2C_2 \int_{[0,T)} \sum_{k=n}^{n+p} |f^k(s)|^2 \,\operatorname{d}\mu_g(s), \; \forall t \in [0,T]
 \end{equation}
 and, from~\eqref{eq:galerkin7}, 
 \begin{equation} \label{eq:galerkin9}
 \sum_{k=n}^{n+p} 
 \int_{[0,T)} \lambda_k |\xi^k(t)|^2\, \operatorname{d} \mu_g(t) \leq 
 2C_1 \sum_{k=n}^{n+p} |u_0^k|^2 +
 2C_2 \int_{[0,T)}\sum_{k=n}^{n+p} |f^k(s)|^2 \,\operatorname{d}\mu_g(s).
 \end{equation}

 \textbf{\textbullet\ Part 3: taking to the limit the discrete problem.}

 Given $n \in \mathbb{N}$, we write $u_n(t):=\sum_{k=1}^n \xi^k(t) w_k$, since 
 $\xi^k \in \mathcal{AC}_g([0,T])\cap \mathcal{BC}_g([0,T])$, 
 we have that
 $u_n \in \mathcal{BC}_g([0,T],L^2(\Omega))$. 
 Thanks to the bound in
\eqref{eq:galerkin8} we observe that $\{u_n\}_{n \in \mathbb{N}}$ is a Cauchy sequence in $\mathcal{BC}_g([0,T],L^2(\Omega))$. Indeed, on one hand,
 \begin{equation}
 \sup_{t \in [0,T]} \|u_n(t)\|_{L^2(\Omega)} = \sup_{t \in [0,T]} \left[
 \sum_{k=1}^n 
 |\xi^k(t)|^2\right]^{1/2},
 \end{equation}
 so, taking into account~\eqref{eq:galerkin8} and the subadditivity of the square root,
 \begin{equation}
 \begin{aligned}
 &\sup_{t \in [0,T]} \|u_{n+p}(t)-u_n(t)\|_{L^2(\Omega)} = 
 \sup_{t \in [0,T]} \left[
 \sum_{k=n}^{n+p} 
 |\xi^k(t)|^2\right]^{1/2} \\
 \leq &\sqrt{2 C_1}\, \left[\sum_{k=n}^{n+p}|u_0^k|^2 \right]^{1/2} +
 \sqrt{2 C_2}\, \left[
 \int_{[0,T)} \sum_{k=n}^{n+p} |f^k(s)|^2 \,\operatorname{d}\mu_g(s)
 \right]^{1/2},
 \end{aligned}
 \end{equation}
 from where we deduce the Cauchy character of the series at the left hand side of the equality. In particular, 
 since $\mathcal{BC}_g([0,T],L^2(\Omega))$ is a Banach space, 
 the sequence will be convergent to an element $u \in \mathcal{BC}_g([0,T],L^2(\Omega))$. Furthermore,
 $u(0)=u_0$. To see this, observe that, since $g$ is continuous at $0$ and $u_n\in\mathcal{BC}_g([0,T],L^2(\Omega))$, $u_n$ are continuous at $0$, 
 so $u_n(0)=u_n(0^+)$.

 From equation~\eqref{eq:galerkin4} and the continuity of $g$ at $0$ we have that
 \begin{equation}
 \begin{aligned}
& \|u_n(0^+)-u_0\|_{L^2(\Omega)}^2 = \sum_{k=1}^{n} |\xi^k(0^+)-u_0^k|^2+\sum_{k=n+1}^{\infty} |u_0^k|^2 \\= &
 \sum_{k=1}^{\infty} \left|(\widehat{e}^{k}(0^+)^{-1}-1)u_{0}^k +
 \widehat{e}^{k}(0)^{-1} \int_{[0,0^+)} \widehat{e}^{k}(s) \widetilde{f}^k(s) \,\operatorname{d}\mu_g(s)\right|^2 \\ & +\sum_{k=n+1}^{\infty}|u_0^k|^2 =  \sum_{k=n+1}^{\infty}|u_0^k|^2 \to 0. 
 \end{aligned}
 \end{equation}
 Hence, $\{u_n(0)\}_{n\in\mathbb N}\to u_0$ and so $u(0)=u_0$. On the other hand, if we take into account that
 \begin{equation}
 \|u_n\|_{L^2_g([0,T],H^1_0(\Omega))}= \left[\int_{[0,T)}
 \sum_{k=1}^n\lambda_k |\xi^k(t)|^2\, \operatorname{d} \mu_g(t)
 \right]^{1/2},
 \end{equation}
 we have that, thanks to~\eqref{eq:galerkin9},
 \begin{equation}
 \begin{aligned}
 &\|u_{n+p}-u_n\|_{L^2_g([0,T],H^1_0(\Omega))} = 
 \left[
 \int_{[0,T)} \sum_{k=n}^{n+p} \lambda_k |\xi^k(t)|^2 \, 
 \operatorname{d} \mu_g(t)
 \right]^{1/2}
 \\
 \leq & \sqrt{2 C_1}\, \left[\sum_{k=n}^{n+p}|u_0^k|^2 \right]^{1/2} +
 \sqrt{2 C_2}\, \left[
 \int_{[0,T)} \sum_{k=n}^{n+p} |f^k(s)|^2 \,\operatorname{d}\mu_g(s)
 \right]^{1/2}.
 \end{aligned}
 \end{equation}
 Thus, $\{u_n\}_{n\in\mathbb N}$ is a Cauchy sequence in the Banach space $L^2_g([0,T],H^1_0(\Omega))$ and so
 $\{u_n\}_{n \in \mathbb{N}}\to u$ in $L^2_g([0,T],H_0^1(\Omega))$. Finally, given
 $n \in \mathbb{N}$ and $1\leq k\leq n$,~\eqref{eq:galerkin1} can 
 be written for every $t \in [0,T]$ and $k =1,\ldots,n$,
 \begin{equation} \label{eq:galerkin10}
 \begin{aligned}
 & (u_n(t),w_k)=  (u_{0,n},w_k)+\int_{[0,t)} (f_n(s),w_k)\,\operatorname{d}\mu_g(s)\\ &
 -k_1 \int_{[0,t)} \int_{\Omega} \nabla u_n(s) \cdot \nabla w_k \, \operatorname{d} x \, \operatorname{d}\mu_g(s)
 -
 \delta
 \int_{[0,t)} \int_{\Omega} u_n(s) w_k \, \operatorname{d} x \, \operatorname{d}\mu_g(s).
 \end{aligned}
 \end{equation}

 Let us fix an element $k<n$ and take $n\to\infty$. We have that, for every $t \in [0,T]$,
 \begin{equation}
 \begin{array}{l}
 \displaystyle
 \lim_{n \to \infty} \left|(u_n(t)-u(t),w_k)\right|\leq  
 \lim_{n \to \infty} \|u_n(t)-u(t)\|_{L^2(\Omega)}=0, \\ 
 \displaystyle
 \lim_{n \to \infty} \left|\int_{[0,t)} (u_n(s)-u(s),w_k) \,\operatorname{d}\mu_g(s) \right| 
 \\ \displaystyle \quad
 
 \leq  
 \lim_{n \to \infty} \int_{[0,t)} \|u_n(s)-u(s)\|_{H^1_0(\Omega)}
 \|w_i\|_{H_0^1(\Omega)}
 \,\operatorname{d}\mu_g(s) = 0,
 \\ 
 \displaystyle
 \lim_{n \to \infty} \left|\int_{[0,t)} (f_n(s)-f(s),w_k)\,\operatorname{d}\mu_g(s) \right| \\
 \displaystyle \quad
  \leq 
 \lim_{n \to \infty} \int_{[0,t)} \|f_n(s)-f(s)\|_{L^2(\Omega)}\,\operatorname{d}\mu_g(s)=0,
 \end{array}
 \end{equation}
 Furthermore, $\lim_{n \to \infty} (u_n(0),w_k) =
 (u_0,w_k)$. Thus, since we can choose $k$ arbitrarily, we deduce, by the density of the system of vectors $\{w_k\}_{k \in \mathbb{N}}$ in $H_0^1(\Omega)$, that
 \begin{equation} \label{eq:galerkin11}
 \begin{aligned}
 (u(t),w)= & (u_{0},w)+\int_{[0,t)} (f(s),w)\,\operatorname{d}\mu_g(s)
 -k_1 \int_{[0,t)} \int_{\Omega} \nabla u(s) \cdot \nabla w \, \operatorname{d} x \, \operatorname{d}\mu_g(s)
 \\ &
 -
 \delta
 \int_{[0,t)} \int_{\Omega} u(s) w \, \operatorname{d} x \, \operatorname{d}\mu_g(s),\; \forall t \in [0,T], \; 
 \; \forall w \in H_0^1(\Omega).
 \end{aligned}
 \end{equation}

 \textbf{\textbullet\ Part 4: bounding with respect to the data.}
 
  On one hand, 
 we have by~\eqref{eq:galerkin6} and~\eqref{eq:galerkin7} that, for every $t\in[0,T]$,
 \begin{equation}
 \begin{aligned}
 & \|u_n(t)\|_{L^2(\Omega)}^2+
 \int_{[0,T)}\|u_n(s)\|^2_{H_0^1(\Omega)} \, 
 \operatorname{d} \mu_g(s) 
 \\\leq &  4 C_1 \|u_0\|^2_{L^2(\Omega)}
 +4 C_2 \|f\|^2_{L^2([0,T],L^2(\Omega))},
 \end{aligned}
 \end{equation}
 so, taking $n\to \infty$, 
 \begin{equation} \label{eq:galerkin15}
 \|u\|_{L^{\infty}_g([0,T],L^2(\Omega))} +
 \|u \|_{L^2_g([0,T],H^1_0(\Omega))}
 \leq 
 \widehat{C}_1 \|u_0\|_{L^2(\Omega)}+
 \widehat{C}_2 \|f\|_{L_g^{2}([0,T],L^2(\Omega))}.
 \end{equation}
 Recovering the $g$-time derivative from~\eqref{eq:galerkin11},
 \begin{equation}
 u_g'=f-k_1\, \Delta u -k_2 \, u \in L_g^{2}([0,T],H^{-1}(\Omega)),
 \end{equation}
 and using the bounds in~\eqref{eq:galerkin15}, we obtain, redefining the constants if necessary, that
 \begin{equation}\label{eq:galerkin16}
 \begin{aligned}
 & \|u\|_{L^{\infty}_g([0,T],L^2(\Omega))} +
 \|u \|_{L^2_g([0,T],H^1_0(\Omega))}+
 \|u_g'\|_{L_g^2([0,T],H^{-1}(\Omega))}
 \\ 
 \leq & \widehat{C}_1 \|u_0\|_{L^2(\Omega)}+
 \widehat{C}_2 \|f\|_{L_g^{2}([0,T],L^2(\Omega))}.
 \end{aligned}
 \end{equation}

 \textbf{\textbullet\ Part 5: uniqueness of solution.} Suppose that there exists 
 another solution of system~\eqref{eq:estado1} 
 $\widehat{u} \in \widetilde{W}^{1,2,2}_g([0,T],H_0^1(\Omega),H^{-1}(\Omega))\cap 
 \mathcal{BC}_g([0,T],L^2(\Omega))$ in the sense of Definition 
~\ref{definition1}. Then,
 \begin{equation} \label{eq:series1}
 \widehat{u}(t)=\sum_{k=1}^{\infty} \widehat{\xi}^k(t)\, w_k,\; \forall t \in [0,T],
 \end{equation}
 where
 \begin{equation}
 \widehat{\xi}^k(t)=\int_{\Omega} \widehat{u}(t)\, w_k \, \operatorname{d} x
 \end{equation}
 and the convergence of series occurring in~\eqref{eq:series1} is considered in 
 $L^2(\Omega)$ for all $t \in [0,T]$ and in $H_0^1(\Omega)$ for $g$-almost 
 all $t \in [0,T]$. 
 To see this, observe that since $\widehat{u}(t)\in L^2(\Omega)$, for all $t \in [0,T]$, and $\{w_k\}_{
 k \in \mathbb{N}}$ is an orthonormal basis of $L^2(\Omega)$, we have that 
 \begin{equation}
 \widehat{u}(t)=\sum_{k=1}^{\infty} \left(\int_{\Omega} \widehat{u}(t) \, w_k \, 
 \operatorname{d} x \right) w_k,
 \end{equation}
 where the convergence is in $L^2(\Omega)$. Now, $\{w_k/\sqrt{\lambda_k}\}_{k=1}^{\infty}$ is a orthonormal basis of 
 $H_0^1(\Omega)$ associated to scalar product in~\eqref{eq:escalar}, 
 so, since $\widehat{u}(t) \in H_0^1(\Omega)$ for $g$-almost all $t \in [0,T]$, we have that
 \begin{equation}
 \begin{aligned}
 \widehat{u}(t)
 =&
 \sum_{k=1}^{\infty} \left(\widehat{u}(t),\frac{w_k}{\sqrt{\lambda_k}}\right) 
 \frac{w_k}{\sqrt{\lambda_k}}=
 \sum_{k=1}^{\infty} \frac{1}{\lambda_k} (\widehat{u}(t),w_k) w_k \\
 =&
 \sum_{k=1}^{\infty} \frac{\lambda_k}{\lambda_k}\left(
 \int_{\Omega} \widehat{u}(t)\, w_k\, \operatorname{d} x \right)\, w_k,
 \end{aligned}
 \end{equation} 
 where the convergence is in $H_0^1(\Omega)$. Now, given elements $t<s$ in $[0,T]$ (analogous for the case $s<t$) 
 and $k \in \mathbb{N}$, we have that
 \begin{equation}
 \begin{aligned}
 \frac{\widehat{\xi}^k(s)-\widehat{\xi}^k(t)}{g(s)-g(t)}=&
 \int_{\Omega} \frac{\widehat{u}(s)-\widehat{u}(t)}{
 g(s)-g(t)} \, w_k \, \operatorname{d} x 
 =\left< \frac{\widehat{u}(s)-\widehat{u}(t)}{
 g(s)-g(t)} , w_k
 \right>_{H^{-1}(\Omega),H_0^1(\Omega)},
 \end{aligned}
 \end{equation}
 thus,
 \begin{equation}
 \begin{aligned}
 & \lim_{s \to t^+} \left|
 \frac{\widehat{\xi}^k(s)-\widehat{\xi}^k(t)}{g(s)-g(t)}-
 \left<\widehat{u}_g'(t),w_k \right>_{H^{-1}(\Omega),H_0^1(\Omega)}
 \right| \\  
 \leq& \lim_{s \to t^+} \left\|
 \frac{\widehat{u}(s)-\widehat{u}(t)}{
 g(s)-g(t)} -\widehat{u}_g'(t)
 \right\|_{H^{-1}(\Omega)} \|w_k\|_{H_0^1(\Omega)}=0.
 \end{aligned}
 \end{equation}
 Where the convergence is a consequence of
$\widehat{u} \in \widetilde{W}_g^{1,2,2}([0,T],H_0^1(\Omega),H^{-1}(\Omega))$. From 
 the previous expression we deduce that 
 \begin{equation} \label{eq:gder1}
 (\widehat{\xi}^k)_g'(t) = 
 \left<\widehat{u}_g'(t),w_k \right>_{H^{-1}(\Omega),H_0^1(\Omega)}\; 
 \text{for $g$-a.e. $t\in [0,T]$}.
 \end{equation}
 Therefore, $\widehat{\xi}^k \in \widetilde{W}_g^{1,2}(0,T) \cap \mathcal{BC}_g([0,T])$. Now, 
 assume $\widehat{u}$ is a solution of system~\eqref{eq:estado1}. 
 Therefore, for $g$-a.e. $t \in [0,T]$ and  every  $k\in \mathbb{N}$,
 \begin{equation}
 \begin{aligned}
 (\widehat{\xi}^k)_g'(t) 
 = &
 \left<\widehat{u}_g'(t),w_k \right>_{H^{-1}(\Omega),H_0^1(\Omega)} 
\\ = &
 \left< f(t) + k_1 \nabla\cdot\nabla \widehat{u}(t) - 
 k_2 \widehat{u}(t),w_k\right>_{H^{-1}(\Omega),H_0^1(\Omega)} 
 \\ 
 = & 
 f^k(t)-(\widehat{u}(t),w_k)
 =f^k(t)-\lambda_k \widehat{\xi}^k(t).
 \end{aligned}
 \end{equation}
 From the previous expression we have that $\widehat{\xi}^k$ satisfies the following equation:
 \begin{equation} 
 \begin{dcases}
 (\widehat{\xi}^k_g)'(t) + \lambda_i \,\widehat{\xi}^k(t)=f^k(t), & 
 \text{for $g$-almost all $t \in [0,T]$}, \\
 \widehat{\xi}^k(0)=u_{0}^k, & k\in \mathbb{N},
 \end{dcases}
 \end{equation}
 which is the same equation that satisfies $\xi^k$ for $k \in \mathbb{N}$, in \eqref{eq:galerkin3}. Hence, by the uniqueness of solution of previous system, 
 we have that $\xi^k(t)=\widehat{\xi}^k(t)$, $t\in [0,T]$, $\forall k \in \mathbb{N}$, 
 and then, $u$ and $\widehat{u}$ are essentially the same element. 
\end{proof}

\begin{rem} We must point out that the solution we have obtained in
~\eqref{eq:galerkin2} is not, in general, continuous at the 
points of $[0,T] \cap D_g$. Indeed, given $t \in [0,T] \cap D_g$, 
 we have that
 \begin{equation}
 \xi^k(t^+) = \xi^k(t) (1-\lambda_k \Delta g(t)) + 
 f^k(t) \Delta g(t).
 \end{equation}
 \end{rem}

 Let us consider now sufficient conditions in order to guarantee the fulfillment of the 
 existence hypotheses (H1)--(H5).
 As we can see from the proof, such conditions are necessary in order to establish some bounds of the partial sums in some spaces. These appear naturally while establishing the bounds that concern the initial condition and source term.

\begin{cor}[Sufficient conditions] Let $u_0\in L^2(\Omega)$, $f \in L^2_g([0,T],L^2(\Omega))$ and
 $g$ a nondecreasing function, continuous in a neighborhood of $t=0$ and left-continuous in $(0,T]$. A sufficient condition for (H2)--(H5)
 to hold is
 \begin{equation}
 \sum_{u \in [s,t)\cap D_g} \frac{\ln|1-\lambda_k \Delta g(u)|}{\lambda_k} 
 < \mu_g([s,t)\setminus D_g), \; \forall 0\leq s< t \leq T,\quad \forall k \in \mathbb{N}.
 \end{equation}
\end{cor}
\begin{proof} On one hand,
 \begin{equation}
 \begin{aligned}
 & e^{-2 \lambda_k \,\mu_g([0,t)\setminus D_g)} 
 \prod_{u \in [0,t) \cap D_g} |1-\lambda_k \Delta g(u)|^2 \\ 
 = & \begin{dcases} e^{-2\lambda_k \mu_g([0,t))}, & t\in [0,t_0),\\
 \exp\left( -2\lambda_k \left[
 \mu_g([0,t)\setminus D_g) - \sum_{u \in [0,t)\cap D_g} 
 \frac{\ln|1-\lambda_k \Delta g(u)|}{\lambda_k}\right]\right) , & t \in (t_0,T],
 \end{dcases}
 \end{aligned}
 \end{equation}
 for any $t_0 \in [0,T]$ such that $g$ is continuous in $[0,t_0)$. 
 Therefore we obtain the bounds in (H2) and (H3). 
 Let us check now what happens with conditions (H4) and (H5). Given $t \in [0,T]$ and
 $s \in [0,t)$, we have that
 \begin{equation}
 \begin{aligned}
 & e^{-2 \lambda_k \,\mu_g([s,t)\setminus D_g)} 
 \frac{\prod_{u \in [s,t) \cap D_g} |1-\lambda_k \Delta g(u)|^2}{
 |1-\lambda_k \Delta g(s)|^2} 
 \\
 = &
 e^{-2 \lambda_k \,\mu_g([s,t)\setminus D_g)} 
 \prod_{u \in (s,t) \cap D_g} |1-\lambda_k \Delta g(u)|^2
 \\
 \leq &\exp\left( {-2 \lambda_k \left[
 \mu_g([s,t)\setminus D_g)-
 \sum_{u \in [s,t) \cap D_g} \frac{\ln|1-\lambda_k \Delta g(u)|}{\lambda_k}
 \right]}\right) .
 \end{aligned}
 \end{equation}
 In particular, we have that
 \begin{equation}
 \begin{aligned}
 & e^{-2 \lambda_k \,\mu_g([s,t)\setminus D_g)} 
 \frac{\prod_{u \in [s,t) \cap D_g} |1-\lambda_k \Delta g(u)|^2}{
 |1-\lambda_k \Delta g(s)|^2} \\
 \leq &
 \begin{dcases}
 e^{-2 \lambda_k \mu_g([s,t))}, & t \in [0,t_0), \\
 \begin{aligned} & \exp\left( -2 \lambda_k \left[
 \mu_g([s,t_0))+\mu_g([t_0,t)\setminus D_g)\right]\right)\\ &
 \cdot\exp\left(2 \lambda_k\sum_{u \in [t_0,t) \cap D_g} \frac{\ln|1-\lambda_k \Delta g(u)|}{\lambda_k} \right)  ,\end{aligned} & \substack{\displaystyle t\in [t_0,T),\\ \displaystyle s \in [0,t_0),}
 \\
 \exp\left( {-2 \lambda_k \left[
 \mu_g([s,t)\setminus D_g)-
 \sum_{u \in [s,t) \cap D_g} \frac{\ln|1-\lambda_k 
 \Delta g(u)|}{\lambda_k}\right]}\right) , & t, s \in [t_0,T),
 \end{dcases}
 \end{aligned}
 \end{equation}
 from where obtain estimations (H4) and (H5). 
\end{proof}

\begin{rem} We can easily extend the results above to the case of Neumann 
 homogeneous boundary conditions:
 \begin{equation} 
 \begin{dcases} 
 u_g' -\nabla \cdot (k_1 \nabla u) + k_2\, u =f(t,x),& \text{in}\; 
 [(0,T)\setminus C_g] \times \Omega, \\ 
 k_1 \nabla u \cdot \mathbf{n}=0,& \text{on}\; (0,T) \times \partial \Omega, 
 \\ 
 u(0,x)=u_0(x),& \text{in}\; \Omega.
 \end{dcases}
 \end{equation}
 In this case, we obtain a solution in the space $\widetilde{W}_g^{1,2,2}([0,T],H^1(\Omega),
 H^1(\Omega)')\cap \mathcal{BC}_g([0,T],L^2(\Omega))$.
\end{rem}

\begin{rem}Observe that in the case of $g(t)=t$, hypothesis 
(H1)-(H5) are trivially satisfied and we recover the classical results 
for the parabolic partial differential equation (\ref{eq:clasical}). So, 
in a certain sense, the theory that we have developed generalizes 
the classical theory for this type of partial differential equations.
\end{rem}

\section{Applications to population dynamics}

In this section we present a possible application of the theory that we have 
developed in the previous section to a silkworm population model based 
on the example presented in \cite[Section 5]{POUSO2018}. In our case, 
we will consider that we have a diffusion term that allows us to study the spatial distribution of the silkworm in an island (for instance, Gran Canaria). Thus, 
 consider the following equation:
\begin{equation} \label{eq:generalizacion}
\begin{dcases} 
 u_g'(t)-\nabla \cdot (\eta \nabla u)=f(t,u(t),u),& \text{in}\; 
[(0,T)\setminus C_g] \times \Omega, \vspace{0.1cm} \\ 
 \eta \nabla u \cdot \mathbf{n}=0,
& \text{on}\; (0,T) \times \partial \Omega, \vspace{0.1cm} \\
u(0,x)=u_0(x),& \text{in}\; \Omega,
\end{dcases} 
\end{equation}
where $\Omega \subset \mathbb{R}^2$, $\eta>0$, $u_0 \in L^2(\Omega)$, 
$f:[0,T] \setminus C_g \times \mathbb{R} \times L^1_{\operatorname{loc}}(\mathbb{R}) \rightarrow 
\mathbb{R}$ is such that
\begin{equation}
f(t, x, \varphi)=\begin{dcases}-c x, & {\text { if } t \in(5 k, 5 k+4), k=0,1,2, \ldots} \\ 
{-x,} & {\text { if } t=5 k+4, k=0,1,2, \ldots} \\ {\lambda \int_{t-5}^{t-1} \varphi(s) \operatorname{d} s,} & {\text { if } t=5(k+1), k=0,1,2, \ldots}\end{dcases}
\end{equation}
with $c>0$ and $\lambda>0$ and $g: t \in [0,\infty) \rightarrow \mathbb{R}$ defined as
\begin{equation}
g(t)=\begin{dcases}{\frac{1}{2} \sqrt{4 t-t^{2}},} & {\text { if } 0 \leqslant t \leqslant 2}, \\ {1,} & {\text { if } 2<t \leqslant 3}, \\ {2-\sqrt{6 t-t^{2}-8},} & {\text { if } 3<t \leqslant 4}, \\ {3,} & {\text { if } 4<t \leqslant 5}, \\
4+g(t-5), & {\text { if } 5 > t}.
\end{dcases}
\end{equation}

A detailed description of the relationship between the previous functions and the life 
cycle of silkworms can be found in \cite{POUSO2018}. If we integrate~\eqref{eq:generalizacion} in 
the whole domain $\Omega$
 and denote by $ \overline{u}(t)=
\int_{\Omega} u(t) \, \operatorname{d} x$, we have
\begin{equation}
\int_{\Omega} u_g'(t)\, \operatorname{d} x -
\int_{\Omega} \nabla \cdot (\eta \nabla u)\, \operatorname{d} x=
\int_{\Omega} f(t,u(t),u)\, \operatorname{d} x,
\end{equation}
where we have assumed that we can interchange the integral with 
the Stieltjes derivative, we recover the $0$-space-dimensional model 
studied in \cite{POUSO2018}:
\begin{equation} \label{eq:modelmean}
\begin{dcases}
\displaystyle
\overline{u}_g'(t)=f(t,\overline{u}(t),\overline{u}), & \text{in}\; 
(0,T)\setminus C_g, \\ 
u(0)=\overline{u}_0.
\end{dcases}
\end{equation}

The mathematical analysis of equation~\eqref{eq:generalizacion} can be done 
utilizing the same techniques that we have used for the general model~\eqref{eq:estado1}. 
That is, we can consider a spectral basis of $H^1(\Omega)$, solve the corresponding 
problem associated to each eigenvalue and finally pass to the limit. We leave the details to the reader and we will focus on the numerical approximation of the model.

We consider a polygonal approximation of Gran Canaria island (the domain $\Omega$), and the following triangulation $\{\tau^k_h\}_{k=1}^{nt}$ of the 
domain:
\begin{figure}[H]
 \begin{subfigure}{.45\textwidth}
 \centering
 \includegraphics[width=.9\linewidth]{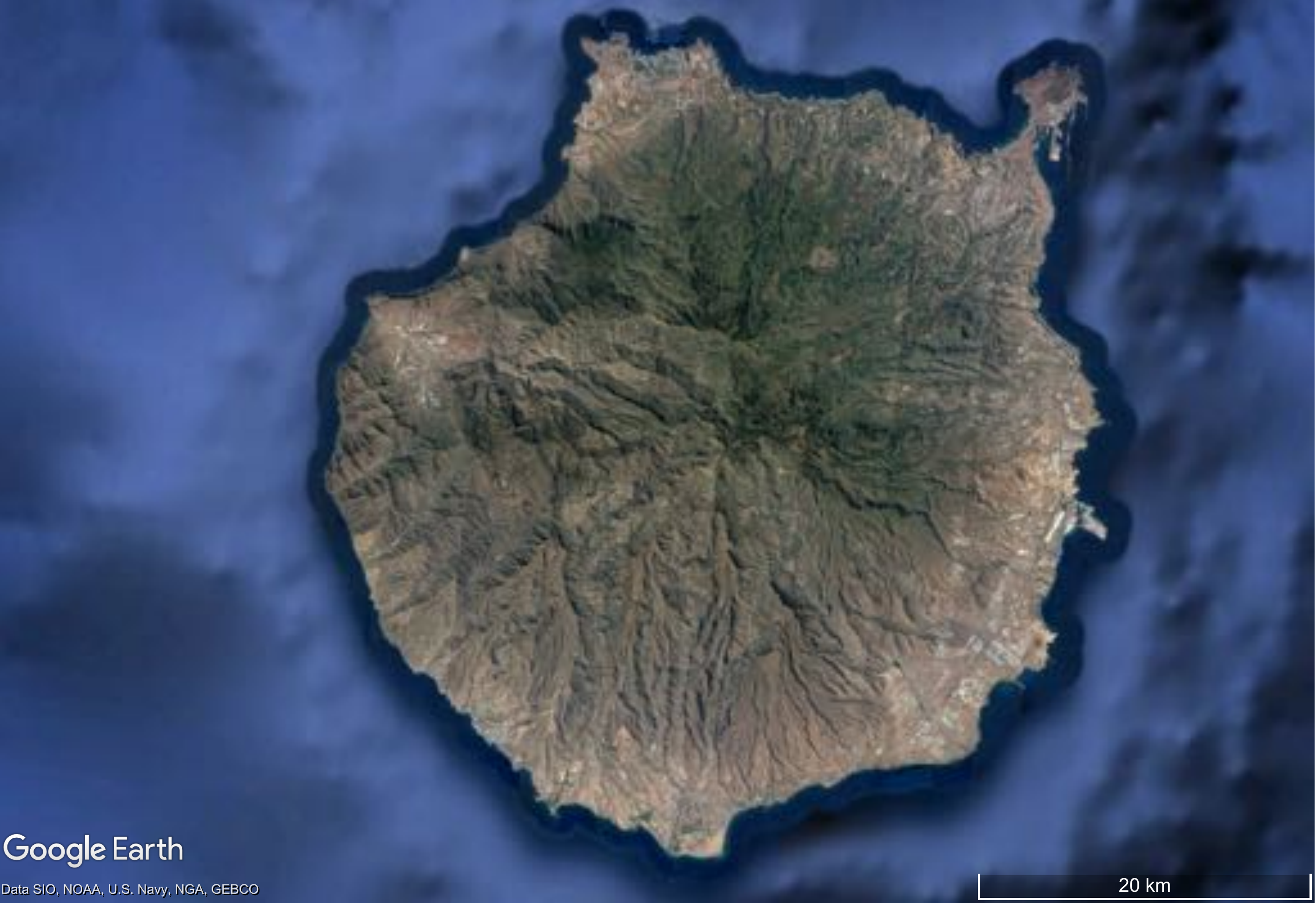}
 \caption{Image of Gran Canaria island generated by Google Earth.}
 \label{fig:sfig1}
 \end{subfigure}\hfill
 \begin{subfigure}{.45\textwidth}
 \centering
 \includegraphics[width=.6\linewidth]{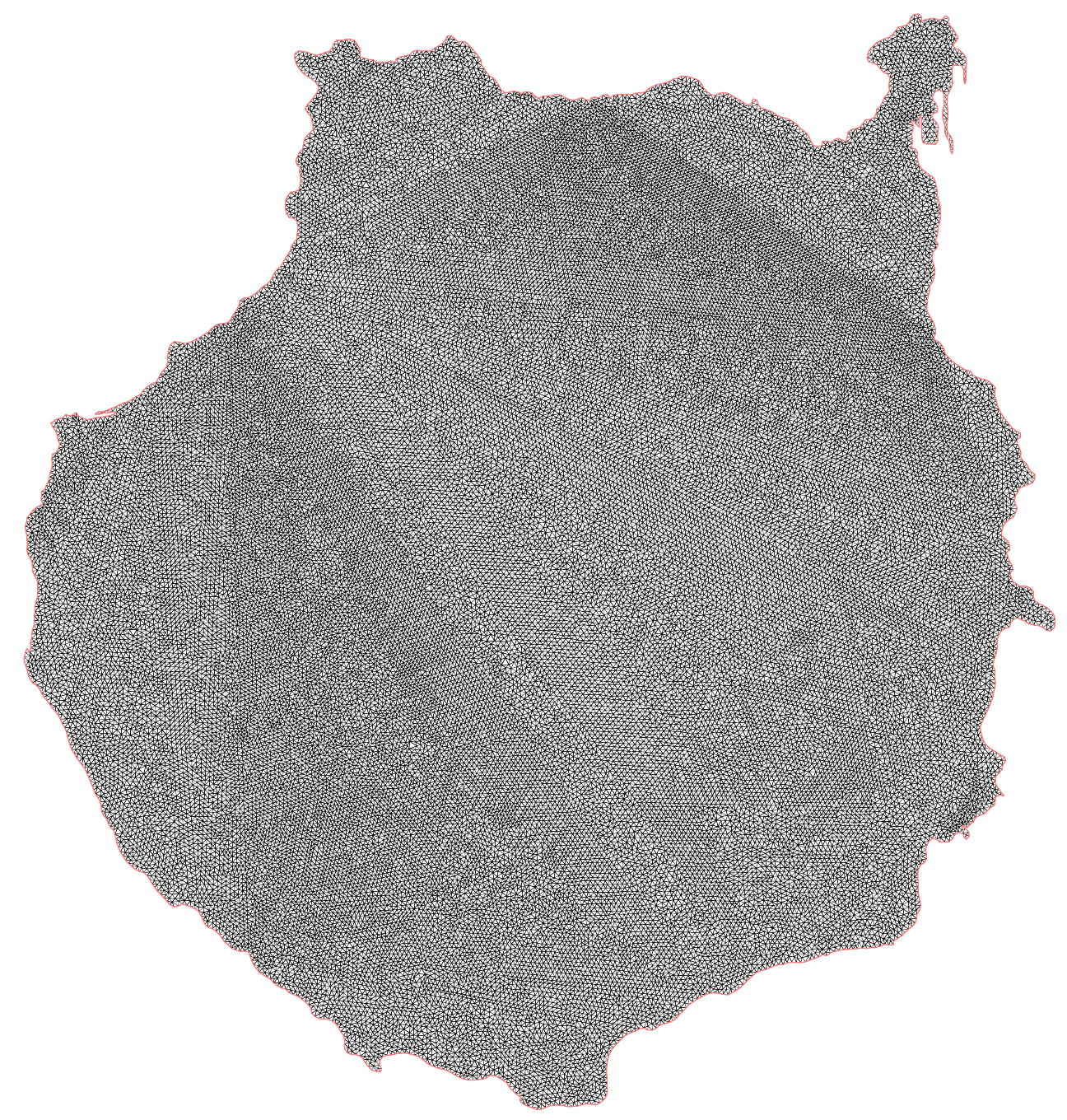}
 \caption{Triangulation generated with FreeFem++\cite{HECHT2012} 
 from the image of the left (88018 triangles).}
 \label{fig:sfig2}
 \end{subfigure}
 \caption{Computational domain.}
 \label{fig:figure1}
\end{figure}
Associated to the previous triangulation, we consider 
the following finite element space:
\begin{equation}
V_h=\{u\in \mathcal{C}(\overline{\Omega}):\; u|_{\tau_h^k} \in 
\mathbb{P}_1(\tau_h^k),\; \forall k=1,\ldots,nt\} \subset H^1(\Omega). 
\end{equation}
Now, let $nv=\dim(V_h)$ (number of vertices) and $\{\varphi_h^j\}_{j=1}^{nv}$ 
a basis of $V_h$ such that
\begin{equation}
\varphi_h^j(e_k)=\delta_{k,j},\; \forall k,\,j\in \{1,\ldots,nv\},
\end{equation}
where $\{e_k\}_{k=1}^{nv}$ are the vertices of the mesh. For $k,j=1,\ldots,nv$ we write
\begin{equation}
\begin{aligned}
[R_h] \in \mathcal{M}_{nv \times nv}({\mathbb R})\ :\ & 
 [R_h]_{k,j}=
\eta \int_{\Omega} \nabla \varphi_k \cdot \nabla \varphi_j \, \operatorname{d} x+\int_{\Omega} \varphi_k \varphi_j \, \operatorname{d} x ,
\vspace{0.1cm}
\\ 
[M_h] \in \mathcal{M}_{nv \times nv}({\mathbb R})\ :\ & 
[M_h]_{k,j}=
\int_{\Omega} \varphi_k \varphi_j \, \operatorname{d} x,.
\end{aligned}
\end{equation}
We will approximate the solution of system~\eqref{eq:generalizacion} by
\begin{equation}
u_h(x,t)=\sum_{k=1}^{nv} \xi^k_h(t) \psi_h^k(x),
\end{equation}
where, for $k=1,\ldots,nv$, $\xi_h^k \in \mathcal{AC}_g([0,T])$ is the solution of
\begin{equation} \label{eq:pbaprox}
\begin{dcases}
 (\xi_h^k)'_g(t) + \lambda_h^k \xi_h^k(t)=f(t,\xi_h^k(t),\xi_h^k) 
,& g\text{-a.e.}\; (0,T)\setminus C_g,\\
\xi_h^k(0)=(u_0,\psi_h^k)_{L^2(\Omega)},
\end{dcases}
\end{equation} 
where 
\begin{equation}
\psi_h^k(x)=\sum_{j=1}^{nv} v^k_j \varphi_j(x)
\end{equation}
with $0<\lambda_h^1\leq \lambda_h^2\leq \cdots\leq \lambda_h^{nv}$ and 
$\{\mathbf{v}^k\}_{k=1}^{nv} \subset 
\mathbb{R}^{nv}$ such that
\begin{equation}
[R_h] \mathbf{v}_h^k = \lambda^k_h [M_h] \mathbf{v}_h^k,\; 
\forall k=1,\ldots,nv.
\end{equation}
We have (see \cite[Theorem 6.4-1]{RAVIART1983}) that 
$ \{\psi_h^k\}_{k=1}^{nv}$ is a basis of $V_h$ orthonormal in 
$L^2(\Omega)$ and $(\psi_h^k,\varphi^j)=\lambda_h^k \left<\psi_h^k,\varphi^j\right>$, $\forall k,\, j =
1,\ldots,nv$, where, given $u,v \in H^1(\Omega)$, 
\begin{equation}
(u,v)=\eta \int_{\Omega} \nabla u \cdot \nabla v \, \operatorname{d} x+\int_{\Omega} u\,v \, \operatorname{d} x.
\end{equation}
Finally, using the same arguments as in \cite{POUSO2018}, we have 
the following expression for the exact solution of~\eqref{eq:pbaprox}:
\begin{equation}
\xi_h^k(t)=\left\{
\begin{array}{ll}
 (u_0,\psi_h^k)\, \mathrm{e}^{-(\lambda_h^k+c-1) g(t)},& 
\text{if } 0\leq t \leq 4, \vspace{0.1cm}\\
 \begin{aligned} & \lambda \int_{5(k-1)}^{5k-1} \xi_h^k(s)\, \operatorname{d} s 
\\ & \cdot \exp\left( -(\lambda_h^k+c-1)[g(t)-g(5k^+)]\right),\end{aligned}  & \text{if } 5k<t\leq 5k+4,\; 
k \in \mathbb{N}, 
\vspace{0.1cm} \\
0, &\text{in other case}.
\end{array}\right.
\end{equation}
In order to implement all the previous approximations we have used 
the software FreeFem++\cite{HECHT2012}. We present some 
results that we have obtaining using the first 150 eigenfunctions and the same 
data as in \cite{POUSO2018}, with
\begin{equation}
u_0(x,y)=\frac{x_0}{\int_{\Omega} (r^2+s^2)\, \operatorname{d} r \,\operatorname{d} s} (x^2+y^2),
\end{equation}
and $T=15$. First, in Figure~\ref{fig:figure2}, we can see a comparison between the 
solution of the $0$-dimensional model and the $2$-dimensional model. We 
observe that the evolution of the spacial mean 
of the $2$-d model solution coincides with the $0$-d model solution, which 
was expected in view of the fact that equation~\eqref{eq:modelmean} has to 
verify the spatial mean of the $2$-d solution. 
\begin{figure}[H]
 \centering
 \includegraphics[width=.8\linewidth]{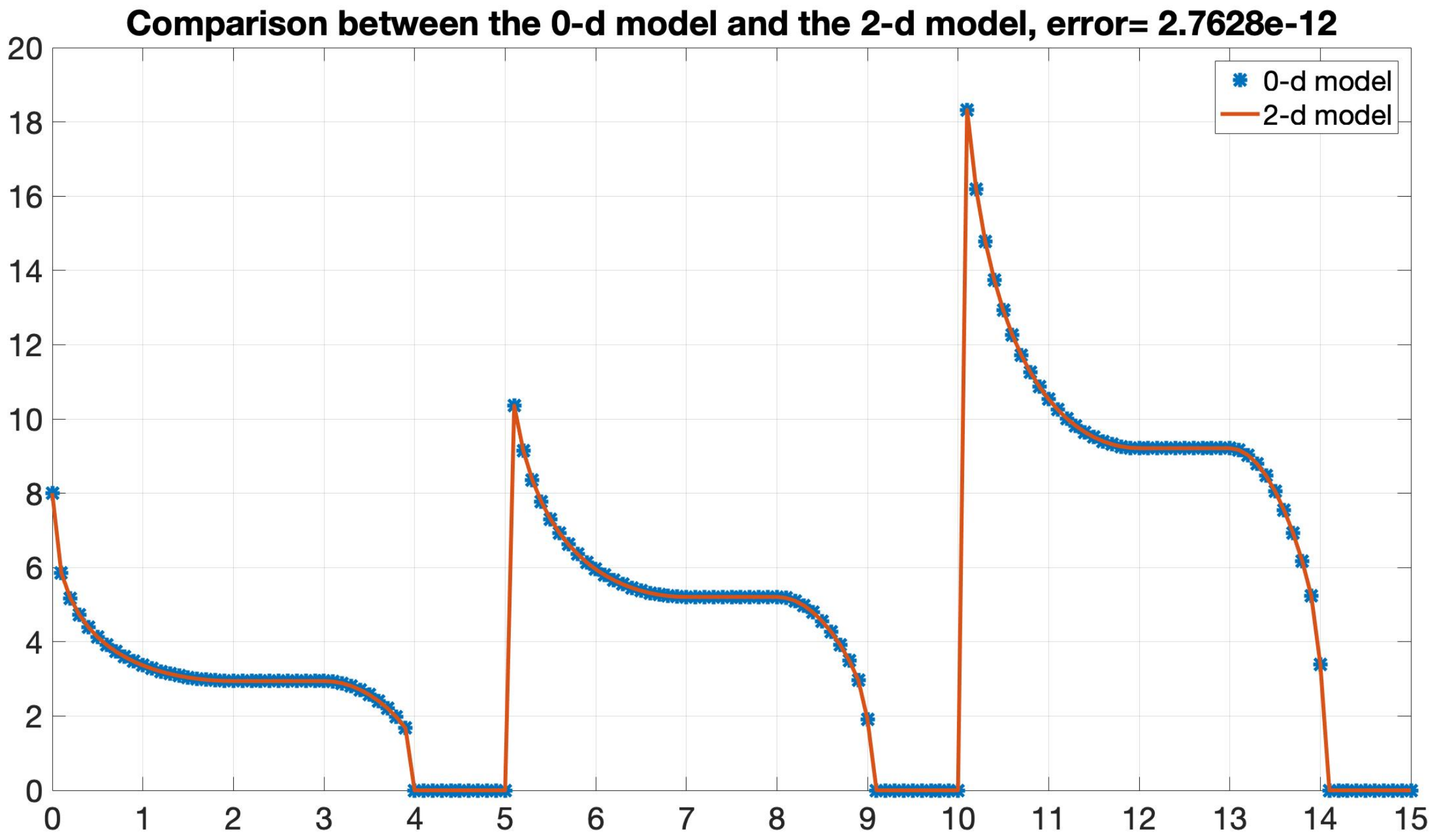}
 \caption{Comparison between the 
 solution of the $0$-d model and the $2$-d model.}
 \label{fig:figure2}
\end{figure}
Secondly, in Figures~\ref{fig:figure3},~\ref{fig:figure4} and~\ref{fig:figure5} we can 
observe, respectively, the initial condition and the solution in the first ($t=5$) and 
second impulse ($t=10$).
\begin{figure}[H]
 \begin{subfigure}{.32\textwidth}
 \centering
 \includegraphics[width=1.\linewidth]{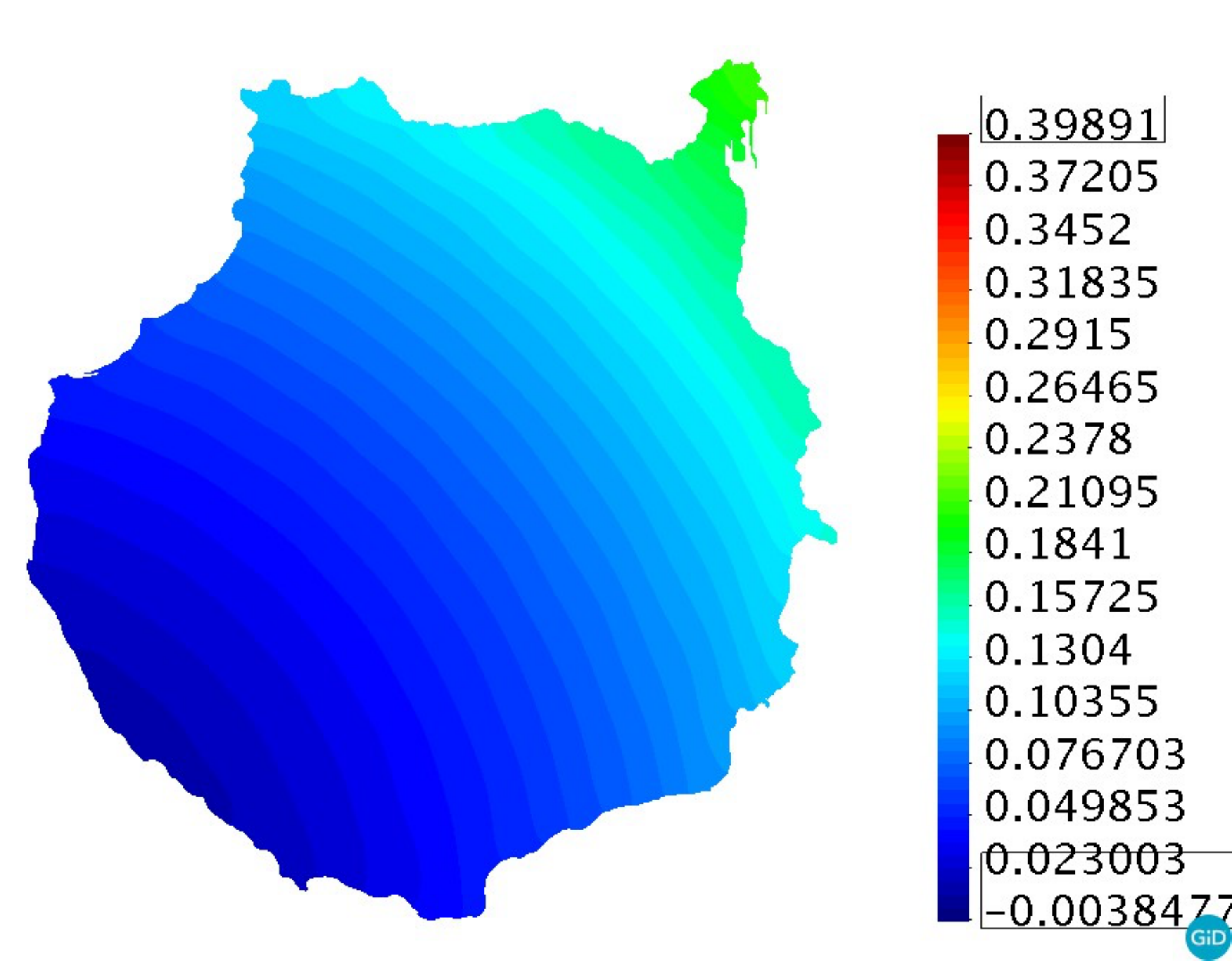}
 \caption{Initial condition.}
 \label{fig:figure3}
 \end{subfigure}
 \begin{subfigure}{.32\textwidth}
 \centering
 \includegraphics[width=1.\linewidth]{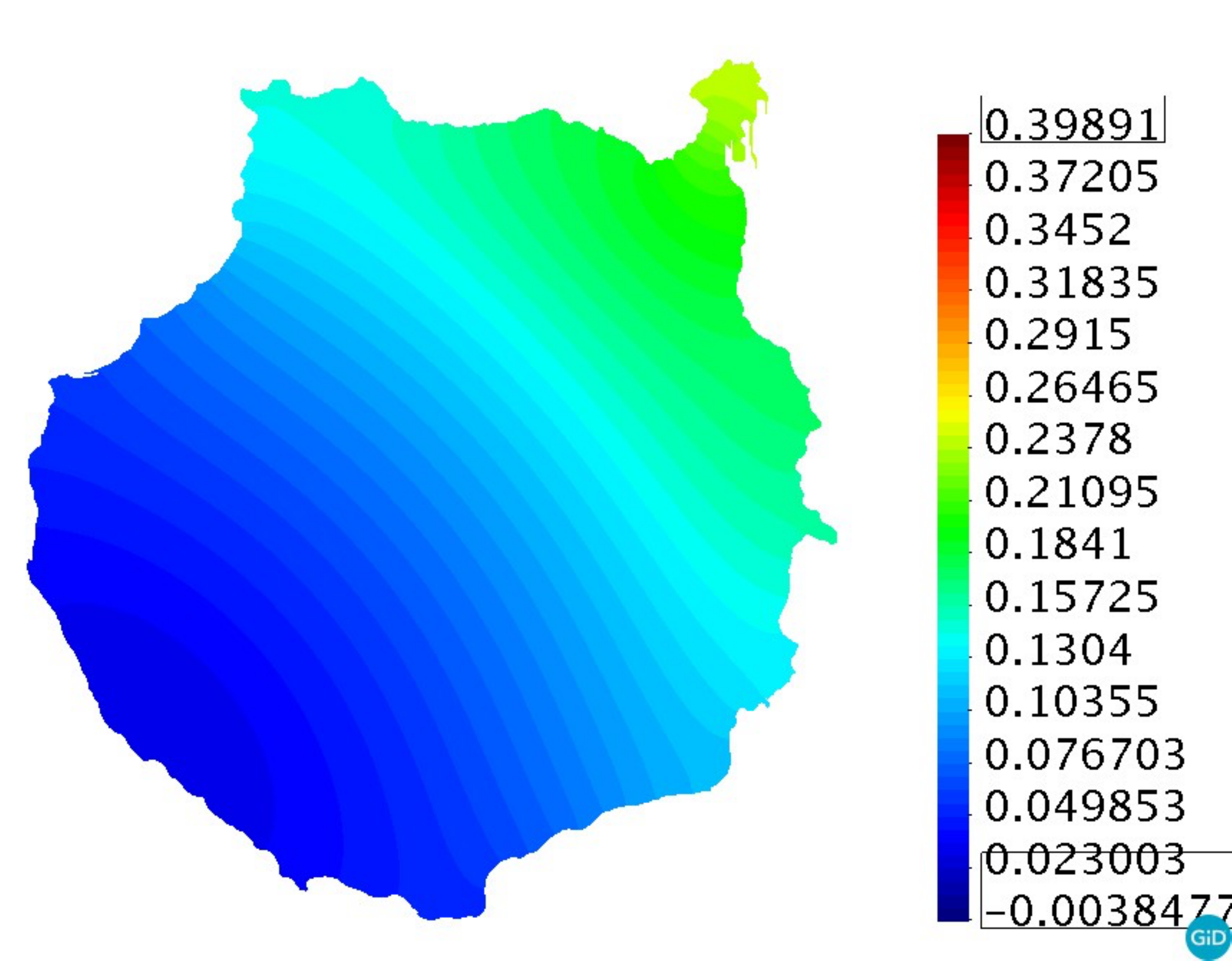}
 \caption{Solution at time $t=5$.}
 \label{fig:figure4}
 \end{subfigure}
\begin{subfigure}{.32\textwidth}
 \centering
 \includegraphics[width=1.\linewidth]{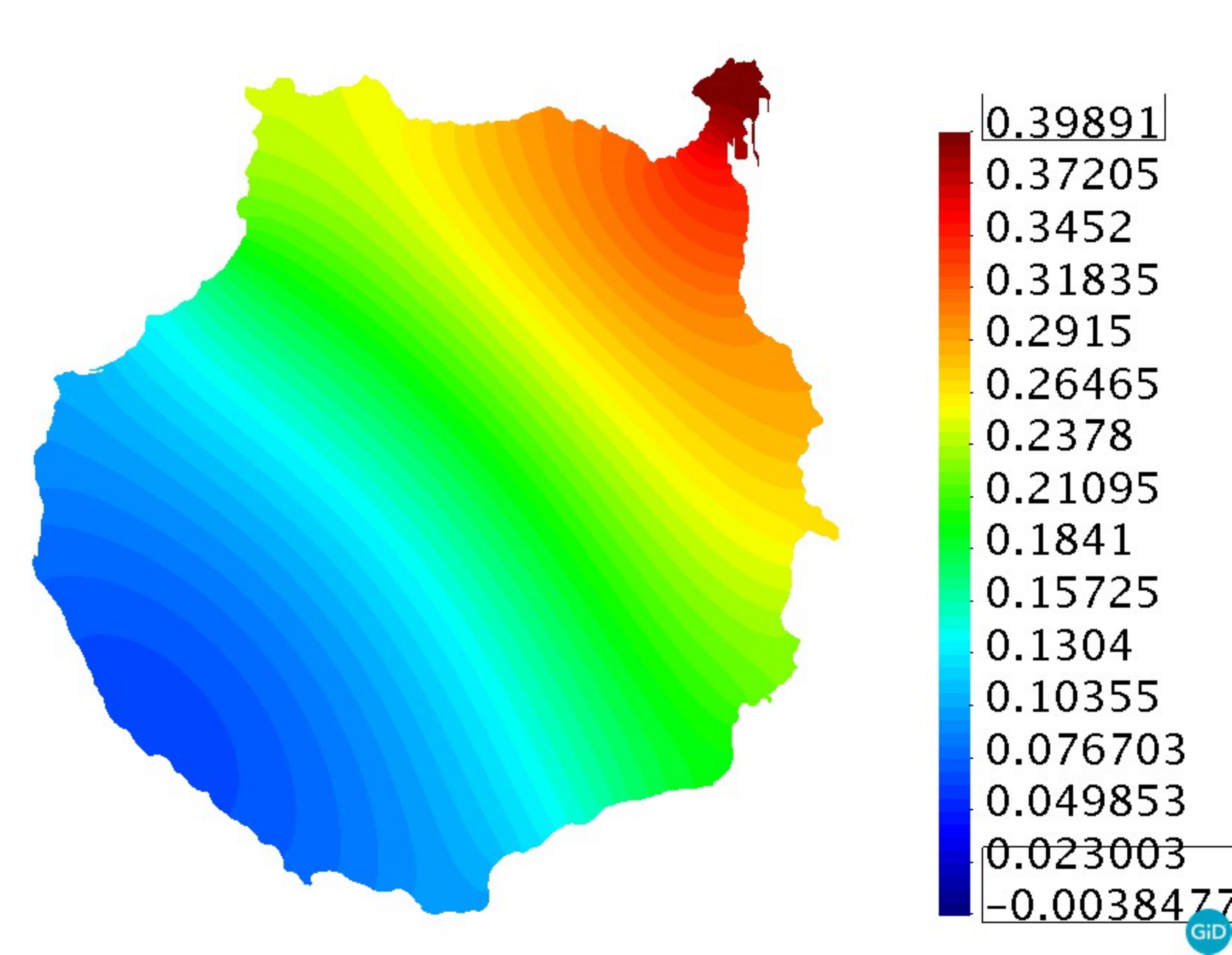}
 \caption{Solution at time $t=10$.}
 \label{fig:figure5}
\end{subfigure}
\caption{Evolution of the model from the initial condition to time $t=10$.}
\end{figure}

\bibliography{Bibliografia}

\end{document}